\documentclass[final,leqno,onefignum,onetabnum]{siamltex1213}

\usepackage{cite}      
\usepackage{graphicx}  
\usepackage{psfrag}    
\usepackage{url}       
\usepackage{amssymb,amsmath,amscd,amsfonts,amstext,mathrsfs,enumerate,xcolor}
\usepackage[shortlabels]{enumitem}
\usepackage{array}
\usepackage{multirow,xcolor}
\usepackage[utf8]{inputenc}
\usepackage[english]{babel} 

\newtheorem{remark}{Remark}
\newtheorem{prop}{Proposition}

\newtheorem{theo}{Theorem}

\newtheorem{coro}{Corollary}
\newtheorem{assum}{Assumption}

\newtheorem{example}{Example}

\newcommand{\sA}{{\mathsf A}}
\newcommand{\sJ}{{\mathsf J}}

\newcommand{\cT}{{\mathcal T}}

\newcommand{\cH}{{\mathcal H}}
\newcommand{\cF}{{\mathcal F}}
\newcommand{\cZ}{{\mathcal Z}}
\newcommand{\cB}{{\mathcal B}}

\newcommand{\cD}{{\mathcal D}}

\newcommand{\cE}{{\mathcal E}}

\newcommand{\cS}{{\mathcal S}}

\newcommand{\bP}{{\mathbb P}}
\newcommand{\bR}{{\mathbb R}}

\newcommand{\bE}{{\mathbb E}}
\newcommand{\E}{{\mathbb E}}

\newcommand{\bN}{{\mathbb N}}

\newcommand{\ps}[1]{\langle #1 \rangle}

\newcommand{\zer}{\mathrm{zer}}
\newcommand{\dom}{\mathrm{dom}}

\newcommand{\graph}{\mathrm{gr}}

\newcommand{\prox}{\mathop{\mathrm{prox}}\nolimits}

\newcommand{\minimax}{\mathop{\mathrm{minimax}}}

\newcommand{\pb}[1]{{\color{black}#1}}
\title{Ergodic convergence of a stochastic \\ proximal point algorithm\thanks{This work was partly supported by the Agence Nationale pour la Recherche, France, (ODISSEE project, ANR-13-ASTR-0030) and the Orange - Telecom ParisTech think tank phi-TAB. Part of this work was published in~\cite{bianchi2015stochastic}.}} 

\author{Pascal Bianchi\thanks{LTCI, CNRS, Télécom ParisTech, Université Paris-Saclay, 75013, Paris, France ({pascal.bianchi@telecom-paristech.fr}).}}

\begin{document}
\maketitle

\begin{abstract}
  The purpose of this paper is to establish the almost sure weak ergodic convergence of
a sequence of iterates $(x_n)$ given by
$$
x_{n+1} = (I+\lambda_n A(\xi_{n+1},\,.\,))^{-1}(x_n)
$$
where $(A(s,\,.\,):s\in E)$ is a collection of maximal monotone operators on a separable Hilbert space,
$(\xi_n)$ is an independent identically distributed sequence of random variables on $E$
and $(\lambda_n)$ is a positive sequence in $\ell^2\backslash \ell^1$. 
The weighted averaged sequence of iterates is shown to converge weakly to a zero (assumed to exist) of the Aumann expectation
${\mathbb E}(A(\xi_1,\,.\,))$ under the assumption that the latter is maximal.
We consider applications to stochastic optimization problems of the form
$$
\min {\mathbb E}(f(\xi_1,x))\ \text{w.r.t. }x\in \bigcap_{i=1}^m X_i
$$
where $f$ is a normal convex integrand and $(X_i)$ is a collection of closed convex sets.
In this case, the iterations are closely related to a stochastic proximal algorithm recently proposed by
Wang and Bertsekas.
\end{abstract}



\pagestyle{myheadings}
\thispagestyle{plain}

\section{Introduction}

The proximal point algorithm is a method for finding a zero of a maximal monotone operator $\sA:\cH\to 2^\cH$
on some Hilbert space $\cH$ \emph{i.e.}, a point $x\in \cH$ such that $0\in \sA(x)$.
The approach dates back to \cite{martinet1970breve}~\cite{rockafellar1976monotone}~\cite{brezis1978produits}
and has aroused a vast literature. The algorithm consists in the iterations
$$
y_{n+1} = (I+\lambda_{n} \sA)^{-1}y_n
$$
for $n\in \bN$ where $\lambda_{n}>0$ is a positive step size. When the
sequence $(\lambda_n)$ is bounded away from zero, it was shown in
\cite{rockafellar1976monotone} that $(y_n)$ converges
weakly to some zero of $\sA$ (assumed to exist).  The case of vanishing
step size was investigated by several authors
including~\cite{brezis1978produits},~\cite{passty1979ergodic}, see
also \cite{alvarez2011unified}. The condition
$\sum_n\lambda_n=+\infty$ is generally unsufficient to ensure the weak
convergence of the iterates $(y_n)$ unless additional assumptions on $\sA$
are made (typically, $\sA$ must be demi-positive).  A counterexample is
obtained when $\sA$ is a $\pi/2$-rotation in the 2D-plane and
$\sum_n\lambda_n^2<\infty$. 
However, the condition $\sum_n\lambda_n=+\infty$ is sufficient to
ensure that $y_n$ converges weakly \emph{in average} to a zero of~$\sA$. Here, by 
\pb{weak convergence in average, or weak ergodic convergence,} we mean that the weighted averaged
sequence
$$
\overline y_n = \frac{\sum_{k=1}^n\lambda_k y_k}{\sum_{k=1}^n\lambda_k}
$$
converges weakly to a zero of $\sA$.

This paper extends the above result to the case where the operator $\sA$
is no longer fixed but is replaced at each iteration $n$ by one operator 
randomly chosen amongst a collection $(A(s,\,.\,): s\in E)$ of maximal monotone operators. 
We study the random sequence $(x_n)$ given by
\begin{equation}
x_{n+1} = (I+\lambda_{n} A({\xi_{n+1}},\,.\,))^{-1}x_n\label{eq:algo-intro}
\end{equation}
where $(\xi_n)$ is an independent identically distributed sequence with
probability distribution $\mu$ on some probability space
$(\Omega,\cF,\bP)$.  We refer to the above iterations as the
stochastic proximal point algorithm.  Under mild assumptions on the
collection of operators, the random sequence $(x_n)$ generated by the
algorithm is shown to be bounded with probability one. The main result
is that almost surely, $(x_n)$ converges weakly in average to some random point
within the set of zeroes (assumed non-empty) of the \emph{mean operator} $\underline A$ defined by
$$
\underline A: x\mapsto  \int A(s,x)d\mu(s)
$$
where $\int$ represents the Aumann integral \cite[Chapter 8]{aubin2009set}.
While the operator $\underline A$ is always monotone,
our key assumption is that it is also maximal. 
This condition is satisfied in a number of particular cases. For instance
when the random variable $\xi_1$ belongs almost surely to a finite set, say $\{1,\dots, m\}$,$
\underline A(x)$ coincides with the Minkowski sum
$$
\underline A(x) = \sum_{i=1}^m \bP(\xi_1 = i)\, A(i,x)
$$
for every $x\in \cH$, and $\underline A$ is maximal under the sufficient condition that the interiors
of the domains of all operators $A(i,\,.\,)$ ($i=1,\dots,m$) have a non-empty intersection~\cite{rockafellar1970maximality}.

\subsection*{Related works and applications}

In the literature,  numerous works have been devoted to iterative algorithms 
searching for zeroes of a sum of maximal operators.
One of the most celebrated approach is the Douglas-Rachford algorithm
analyzed by \cite{lions1979splitting}. Though suited to a sum of two operators,
the Douglas-Rachford algorithm can be adapted to an arbitrary finite sum using the so-called product space trick.
The authors of \cite{brezis1978produits} and \cite{passty1979ergodic} consider applying
product of resolvents in a cyclic manner. Numerically, the above deterministic approaches 
become difficult to implement when the number of operators in the sum is large,
or a fortiori infinite (\emph{i.e.} the mean operator is an integral).
In parallel, stochastic approximation techniques
have been developped in the statistical literature to find a root of an integral functional $h:\cH\to\cH$
of the form $h(x) = \int H(s,x)d\mu(s)$.
The archetypal  algorithm writes $x_{n+1} = x_n-\lambda_{n} H(\xi_{n+1},x_n)$ 
as proposed in the seminal work of Robbins and Monro \cite{robbins1951stochastic}.
It turns out that the iterates~(\ref{eq:algo-intro}) have a similar form
$$x_{n+1} = x_n-\lambda_{n} A_{\lambda_{n}}(\xi_{n+1},x_n)$$ where $A_\lambda(s,\,.\,)$
is the so-called Yosida approximation of the monotone operator $A(s,\,.\,)$.
As a matter of fact, our analysis borrows some proof ideas from 
the stochastic approximation literature~\cite{andrieu2005stability}.

Applications of stochastic approximation include the minimization of integral functionals
of the form $x\mapsto \bE(f(\xi_1,x))$ where $(f(s,\,.\,):s\in E)$ is a collection of proper lower-semicontinuous
convex functions on $\cH\to (-\infty,+\infty]$. 
\pb{We refer to \cite{nemirovski2009robust} or to \cite{bertsekas2010survey} for a survey.
In particular, the benefits in terms of convergence rate of considering
average iterates $\bar x_n = \sum_{k\leq n}\gamma_kx_k/\sum_{k\leq n}\gamma_k$ 
is established by \cite{nemirovski2009robust} in the context of convex programming
and in \cite{juditsky2011solving} in the context of variational inequalities.
Averaging of the iterates is introduced in these works (see also \cite{atc-for-mou-14} for more recent results)
where improved complexity results is the main motivator.}
For instance, the stochastic subgradient algorithm writes $x_{n+1} = x_n-\lambda_{n} \tilde \nabla f(\xi_{n+1},x_n)$
where $\tilde \nabla f(\xi_{n+1},x_n)$ represents a subgradient of $f(\xi_{n+1},\,.\,)$
at point $x_n$ (assumed in this case to be everywhere well defined). The algorithm is often analyzed under
a uniform boundedness assumption of the subgradients \cite{nemirovski2009robust}, \cite{bertsekas2010survey}.
In practice, a reprojection step is often introduced to enforce the boundedness of the iterates.

Denoting by $A(s,\,.\,)$ the subdifferential of $f(s,\,.\,)$,
the resolvent $(I+\lambda A(s,\,.\,))^{-1}$ coincides with the proximity operator associated with $f(s,\,.\,)$
given by
\begin{equation}
\prox_{\lambda f(s,\,.\,)}(x) = \arg\min_{t\in \cH}  \lambda f(s,t)+ \frac{\| t-x\|^2}2
\label{eq:proximity-op}
\end{equation}
for any $x\in \cH$. The iterations~(\ref{eq:algo-intro}) can be equivalently written as
\begin{equation}
x_{n+1} = \prox_{\lambda_{n} f(\xi_{n+1},\,.\,)}(x_n)\,. \label{eq:intro-ite-prox}
\end{equation}
A related algorithm is studied (among others) by Bertsekas in \cite{bertsekas2011incremental}
under the assumption that $\xi_1$ has a finite range and $f(s,\,.\,)$ is defined
on $\bR^d\to\bR$. As functions are supposed to have full domain, 
\cite{bertsekas2011incremental} introduces a projection step onto a closed convex set
in order to cover the case of constrained minimization.
When there exists a constant $c$ such that
the functions  $f(s,\,.\,)$ are $c$-Lipschitz continuous for all $s$, and under other technical assumptions,
the algorithm of \cite{bertsekas2011incremental} is proved to converge to a sought minimizer.
In \cite{wang2013incremental}, the finite range assumption is dropped and random projections are introduced.
Extension to variational inequalities is considered in~\cite{wang2014incremental} (see also the discussion below).

\pb{An important aspect is related to the analysis of the convergence rates of the iterates~(\ref{eq:intro-ite-prox}).
The working draft~\cite{ryu2014stochastic} was brought to our knowledge 
during the review process of this paper.
The authors analyze a related algorithm and provide asymptotic convergence rates in the case 
where the monotone operators $A(s,\,.\,)$ are gradients of convex functions in $\bR^n$ and
assuming moreover that these functions have the same domain, are all strongly convex and twice differentiable.
}

\pb{In order to illustrate~(\ref{eq:algo-intro}), 
we provide some application examples without insisting on the hypotheses for the moment.

The simplest application example correspond to the following feasibility problem: given a collection
of closed convex sets $X_1,\dots, X_m$, find a point $x$ in their intersection 
$X=\bigcap_{i=1}^m X_i$.
The interest lies in the case where $X$ is not known but revealed through random
realizations of the $X_i$'s, so that a straightforward projection onto $X$ is unaffordable \cite{nedic2011random}, \cite{bauschke1996projection}.
The algorithm~(\ref{eq:intro-ite-prox}) encompasses this case by letting $f(\xi_{n+1},\,.\,)$ coincide 
with the indicator function $\iota_{X_{\xi_{n+1}}}$ of the set $X_{\xi_{n+1}}$ (equal to zero on that set and to $+\infty$ elsewhere), 
where $\xi_{n+1}$ is randomly chosen in the set $E=\{1,\dots,m\}$ according to some distribution
$\mu=\sum_{i=1}^m\alpha_i \delta_i$ where all the $\alpha_i$'s are positive and $\delta_i$ is the Dirac measure at $i$.
In this case, the algorithm~(\ref{eq:intro-ite-prox}) boils down to a special case of~\cite{nedic2011random} 
and consists in successive projections onto randomly selected sets. The algorithm is of particular interest when $m$ is large 
(our framework even encompasses the case of an infinite number of sets) or in the case of distributed optimization methods: in that case, $X_i$ is 
the set of local constraints of an agent $i$ and $X$ is nowhere observed \cite{bertsekas2010survey}.
As pointed out in \cite{nedic2011random}, examples of applications include fair rate allocation problems in wireless networks 
where $X_i$ represent a set of channel states~\cite{eryilmaz2005fair}, \cite{huang2009joint}, \cite{stolyar2005asymptotic}
or image restoration and tomography \cite{capricelli2007convex}, \cite{combettes1997convex}.

A generalization of the above feasibility problem is the programming problem
\begin{equation}
\min_x F(x)\ \text{ s.t. }\ x\in \bigcap_{i=1}^m X_i\label{eq:min-cons}
\end{equation}
where $F$ is a closed proper convex function. Here we set $\tilde
f(0,\,.\,)=F$ and $\tilde f (i,\,.\,)=\iota_{X_i}$ for $1\leq i\leq m$
and choose randomly the variable $\xi_{n+1}$ on the set $E=\{0,1\dots
m\}$ according to some discrete distribution $\sum_{i=0}^m
\tilde\alpha_i\delta_i$ for some positive coefficients $\tilde
\alpha_i$.  The use of algorithm~(\ref{eq:intro-ite-prox}) with $f$
replaced by $\tilde f$ leads to an algorithm where either
$\prox_{\lambda_n F}$ is applied to the current estimate or a
projection onto one of the sets $X_i$ is done, depending on the outcome
of $\xi_{n+1}$.  A refinement consists in assuming that the function
$F$ is itself an expectation of the form $F(x) = \bE(f(Z,x))$ for some
random variable $Z$.  In this case, the previous algorithm can be
extended by substituting $\prox_{\lambda_n F}$ with a random version
$\prox_{\lambda_n f(Z_{n+1},\,.\,)}$ where  $(Z_n)_n$ are
iid copies of $Z$. This example will be discussed in details in
Section~\ref{sec:appli}.

Apart from convex minimization problems, Algorithm (\ref{eq:algo-intro}) also finds applications in
minimax problems \emph{i.e.,} when the aim is to search for a saddle point of a given function $L$ 
\cite{cass1976structure}, \cite{rockafellar1970monotone}. 
Suppose that $\cH$ is a cartesian product of two Hilbert spaces $\cH_1\times\cH_2$
and define $\ell:E\times\cH\to [-\infty,+\infty]$ such that $\ell(s,x,y)$ is convex in $x$ and concave in $y$
and $\ell(s,\,.\,)$ is proper and closed in the sense of \cite{rockafellar1970monotone}.
Consider the problem of finding a saddle point $(x,y)$ of function $L = \bE(\ell(\xi_1, \,.\,))$
\emph{i.e.} $(x,y)\in \arg\minimax L$. For every $s\in E$ and $z\in \cH$ of the form $z=(x,y)$,
define $A(s,x,y)$ as the set of points $(u,v)$ such that for every $(x',y')$,
$$
\ell(s,x',y)-\ps{u,x'}+\ps{v,y}\geq 
\ell(s,x,y)-\ps{u,x}+\ps{v,y}\geq
\ell(s,x,y')-\ps{u,x}+\ps{v,y'}\,.
$$
In that case, the operator $A(s,\,.\,)$ is maximal monotone for every $s$, and
the stochastic proximal point algorithm~(\ref{eq:algo-intro}) reads 
$$
(x_{n+1},y_{n+1}) = \arg\minimax_{(x,y)}\ \ell(\xi_{n+1}, x,y)+\frac{\|x-x_n\|^2}{2\lambda_n}-\frac{\|y-y_n\|^2}{2\lambda_n}\,.
$$ 

As a further extension, Algorithm~(\ref{eq:algo-intro}) can be used to solve variational inequalities.
Let $X=\cap_{i=1}^m X_i$ be defined as above and consider the problem of finding $x^\star\in X$ such that
\begin{equation}
\forall x\in X,\ \ps{F(x^\star),x-x^\star}\geq 0\label{eq:VI}
\end{equation}
where $F:\cH\to\cH$ is monotone and, for simplicity, single-valued (extension to set-valued $F$
is also possible in our framework).  Applications of~(\ref{eq:VI}) are numerous. We refer to  \cite{kin-sta-(livre)00} for an overview.
Specific applications include game theory where typically, a Nash equilibrium has to be found amongst
users having individual constraints and observing possibly stochastic rewards \cite{scutari2014real}.
Other examples such as matrix minimax problems are described in \cite{juditsky2011solving}.
Similarly to the programming problem~(\ref{eq:min-cons}),
the application of the stochastic proximal point algorithm to the variational inequality~(\ref{eq:VI}) yields the following algorithm.
Depending on the outcome of a random variable $\xi_{n+1}\in \{0,\dots,m\}$, a projection
onto one of the sets $X_1,\dots,X_m$ is performed, or the resolvent $(I+\lambda_n F)^{-1}$
is applied to the current estimate. 

Also interesting is the case where the function $F$ in~(\ref{eq:VI}) is itself defined
as an expectation of the form $F(x)=\bE(f(Z,x))$ where $f$ is $\cH$-valued and $Z$ is a r.v.
In this case, the previous algorithm can be generalized by substituting the resolvent $(I+\lambda_n F)^{-1}$
with its stochastic counterpart $(I+\lambda_n f(Z_{n+1},\,.\,))^{-1}$ where $(Z_n)_n$ are iid copies of $Z$.
The context of stochastic variational inequalities 
is investigated by Juditsky \emph{et al.}, see \cite{juditsky2011solving} where a stochastic mirror-prox algorithm
is provided. The algorithm of \cite{juditsky2011solving}
uses general prox-functions and allows for a possible bias in the estimation of $F$.
In \cite{juditsky2011solving}, $X$ is supposed to be a compact subset of $\bR^N$, $m$ is equal to one, and $\|F(x)-F(y)\|_*\leq L\|x-y\|+M$
(for some arbitrary norm $\|\,.\,\|$ and the corresponding dual norm  $\|\,.\,\|_*$)  where
$L$, $M$ are constants that are known by the user. Moreover, a variance bound of the form
$\bE(\|f(Z,x)-F(x)\|^2)\leq \sigma^2$ is supposed to hold uniformly in $x$.
Then, using a constant step size depending on $L$, $M$
and the expected number of iterations of the algorithm, the authors prove that the algorithm
achieves optimal convergence rate. Note that the \emph{black-box} model used in the present paper
is different from \cite{juditsky2011solving} in the sense that we are making an implicit use of 
$f(Z_{n+1},\,.\,)$ instead of an explicit one as in \cite{juditsky2011solving}. In our work, 
this permits to prove the almost sure convergence of the algorithm under weaker assumptions than \cite{juditsky2011solving}.
On the other side, the price to pay with our approach is the absence 
of convergence rate certificates.
Also related to our framework is the recent work \cite{wang2014incremental}.
An algorithm similar to ours is proposed,
$F$ being moreover assumed to be strongly monotone and to verify the lipschitz-like property
$\E(\|f(Z,x)-f(Z,y)\|^2)\leq C\|x-y\|^2$.
These assumptions are not needed in our approach.

}

\subsection*{Organization and contributions} The paper is organized as follows.
After some preliminaries in Section~\ref{sec:preliminaries},
the main algorithm is introduced in Section~\ref{sec:algo}.
The aim of Section~\ref{sec:stability} is to establish that the algorithm is stable in the sense
that the sequence $(x_n)$ is bounded almost surely. We actually prove a stronger result: 
for any zero $x^\star$ of $\underline A$, the sequence $\|x_n-x^\star\|$ converges almost surely.
This point is the first key element to prove the weak convergence in average of the algorithm.
The second element is provided in Section~\ref{sec:cv} where it is shown that
any weak cluster point of the weighted averaged sequence $(\overline x_n)$ is a zero of $\underline A$.
Putting together these two arguments and using Opial's lemma~\cite{passty1979ergodic},
we conclude that, almost surely, $(\overline x_n)$ converges weakly to a zero of $\underline A$.
The proofs of Section~\ref{sec:cv} rely on two major assumptions. First, the operator 
$\underline A$ is assumed maximal, as discussed above. Second, the averaged sequence
of (random) Yosida approximations evaluated at the iterates is supposed to be uniformly integrable 
with probability one. The latter assumption is easily verifiable 
when all operators are supposed to have the same domain. 
The case where operators have different domains is more involved.
We introduce a linear regularity assumption of the set of domains
of the operators inspired by~\cite{bauschke1996projection}
(a similar assumption is also used in \cite{wang2013incremental}).
 We provide estimates of the distance 
between the iterate $x_n$ and the essential intersection of the domains.
The latter estimates allow to verify the uniform integrability condition, and
yield the almost sure weak convergence in average of the algorithm in the general case.

In Section~\ref{sec:appli}, we study applications to convex programming. We use our results 
to prove weak convergence in average of $(x_n)$ given by~(\ref{eq:intro-ite-prox})
to a minimizer of $x\mapsto \bE(f(\xi_1,x))$. As an illustration, we address the problem
$$
\min\  \bE(f(\xi_1,x)) \ \text{w.r.t. }{x\in \bigcap_{i=1}^m X_i}
$$
where $X_1,\dots,X_m$ are closed convex sets of $\bR^d$ and $f(s,\,.\,)$ is a convex function
on $\cH\to\bR$ for each $s\in E$. We propose a random algorithm quite 
similar to~\cite{wang2013incremental} and whose 
convergence in average can be established under verifiable conditions.

\section{Preliminaries}
\label{sec:preliminaries}

\subsection*{Random closed sets} 

Let $\cH$ be a separable Hilbert space (identified with its dual) equipped with its Borel $\sigma$-algebra $\cB(\cH)$.
We denote by $\|x\|$ the Euclidean norm of any $x\in \cH$ and by ${ d}(x,Q)  = \inf\{\|y-x\| : y\in Q\}$
the distance between a point $x\in \cH$ and a set $Q\in 2^\cH$ (equal to $+\infty$ when $Q=\emptyset$).
We denote by $\mathrm{cl}(Q)$ the closure of $Q$.
We note $|Q| = \sup\{\|x\| : x\in Q\}$. 

Let $(T,\cT)$ be a measurable space. Let $\Gamma:T\to 2^\cH$ be a multifunction such that
$\Gamma(t)$ is a closed set for all $t\in T$. The domain of $\Gamma$ is denoted by
$\dom(\Gamma)=\{t\in T:\,\Gamma(t)\neq\emptyset\}$. The graph of $\Gamma$ is denoted by
$\graph(\Gamma) = \{(t,x)\,:\,x\in \Gamma(t)\}$.

We say that $\Gamma$ is $\cT$-measurable (or Effros-measurable)
if $\{ t\in T:\,\Gamma(t)\cap U\neq \emptyset\}\in \cT$ for each open set $U\subset \cH$.
This is equivalent to say that for any $x\in \cH$, the mapping $t\mapsto { d}(x,\Gamma(t))$ 
is a random variable \cite{castaing1977convex}, \cite{molchanov}.
We say that $\Gamma$ is graph-measurable if $\graph(\Gamma)\in \cT\otimes \cB(\cH)$.
Effros-measurability implies graph measurability and the converse is true
if $(T,\cT)$ is complete for some $\sigma$-finite 
measure~\cite[Chapter III]{castaing1977convex}, \cite[Theorem 2.3, pp.28]{molchanov}.

Given a probability measure $\nu$ on $(T,\cT)$,
a function $\phi:T\to\cH$ is called a measurable selection of $\Gamma$ if 
$\phi$ is $\cT/\cB(\cH)$-measurable and if $\phi(t)\in \Gamma(t)$ for all $t$ $\nu$-a.e.
We denote by $\cS(\Gamma)$ the set of measurable selections of $\Gamma$.
If $\Gamma$ is measurable, the measurable selection theorem states that 
$\cS(\Gamma)\neq\emptyset$ if and only 
if $\Gamma(t)\neq \emptyset$ for all~$t$~$\nu$-a.e. \cite[Theorem 2.13, pp.32]{molchanov}, \cite[Theorem 8.1.3]{aubin2009set}.
For any $p\geq 1$, we denote by $L^p(T,\cH,\nu)$ the set of measurable
functions $\phi:T\to\cH$ such that $\int \|\phi\|^pd\nu<\infty$.
We set $\cS^p(\Gamma) = \cS(\Gamma)\cap L^p(T,\cH,\nu)$.
The Aumann integral of the measurable map $\Gamma$ is the set 
$$
\int \Gamma d\nu = \left\{\int \phi\, d\nu\,:\,\phi\in \cS^1(\Gamma)\right\}
$$
where $\int \phi\, d\nu$ is the Bochner integral of $\phi$.

\subsection*{Monotone operators}

An operator $\sA:\cH\to 2^\cH$ is said monotone if $\forall (x,y)\in
\graph(\sA)$, $\forall (x',y')\in \graph(\sA)$, $\ps{y-y',x-x'}\geq 0$.
\pb{It is said strongly monotone with modulus $\alpha$ if the inequality
$\ps{y-y',x-x'}\geq 0$ can be replaced by  
$\ps{y-y',x-x'}\geq \alpha \|x-x'\|^2$.}
The operator $\sA$ is {maximal monotone} if it is monotone and if for
any other monotone operator $\sA':\cH\to 2^\cH$, $\graph (\sA)\subset
\graph(\sA')$ implies $\sA=\sA'$. A maximal monotone operator $\sA$
has closed convex images and $\graph(\sA)$ is closed
\cite[pp. 300]{bauschke2011convex}.
We denote the identity by $I:x\mapsto x$.  For some $\lambda>0$, the
resolvent of $\sA$ is the operator $\sJ_\lambda = (I+\lambda
\sA)^{-1}$ or equivalently: $y\in \sJ_\lambda(x)$ if and only if
$(x-y)/\lambda \in \sA(y)$. The Yosida approximation of $\sA$ is the
operator $\sA_\lambda = (I-\sJ_\lambda)/\lambda$. Assume from now on
that $\sA$ is a maximal monotone operator. Then $\sJ_\lambda$ is a
single valued map on $\cH\to\cH$ and is firmly non-expansive in the
sense that $\ps{\sJ_\lambda(x)-\sJ_\lambda(y),x-y}\geq
\|\sJ_\lambda(x)-\sJ_\lambda(y)\|^2$ for every $(x,y)\in \cH^2$.  The
Yosida approximation $\sA_\lambda$ is $1/\lambda$-Lipschitz continuous
and satisfies $\sA_\lambda(x)\in \sA(\sJ_\lambda(x))$ for every $x\in
\cH$ \cite{minty1962monotone}, \cite[Corollary 23.10]{bauschke2011convex}. 
For any $x\in \dom(\sA)$, we denote by
$\sA_0(x)$ the element of least norm in $\sA(x)$ \emph{i.e.},
$\sA_0(x)=\mathrm{proj}_{\sA(x)}(0)$ where $\mathrm{proj}_C$
represents the projection operator onto a closed convex set $C$.  When
$\sA$ is maximal monotone and $x\in \dom(\sA)$, then
$\|\sA_\lambda(x)\| \leq \|\sA_0(x)\|$.  In that case,
$\sA_\lambda(x)$ and $\sJ_\lambda (x)$ respectively converge to
$\sA_0(x)$ and $x$ as $\lambda\downarrow 0$ \cite[Section 23.5]{bauschke2011convex}.

\pb{ 
\subsection*{Random convex functions}

A function $f:E\times \cH\to (-\infty,+\infty]$ is called a normal convex integrand if it 
is $\cE\otimes \cB(\cH)$-measurable and if $f(s,\,.\,)$ 
is lower semicontinuous proper and convex for each $s\in E$ \cite{rockafellar1971convex}.
For such a function $f$, we define 
\begin{equation}
F(x) =\int f(s,x)d\mu(s)\label{eq:F}
\end{equation}
where the above integral is defined as the sum
$$
\int f(s,x)^+ d\mu(s) - \int f(s,x)^- d\mu(s) 
$$
where we use the notation $a^\pm= \max(\pm a,0)$ and the convention $(+\infty)-(+\infty)=+\infty$.
The subdifferential operator $\partial f:E\times \cH\to\cH$ is defined for all $(s,x)\in E\times\cH$ by
$$\partial f(s,x) = \{u\in \cH\,:\forall y\in \cH, f(s,y)\geq f(s,x)+\ps{u,y-x}\}\,.$$
}

\section{Algorithm}
\label{sec:algo}

\subsection{Description}

Let $(E,\cE,\mu)$ be a complete probability space and let $\cH$ 
be a separable Hilbert space equipped with its Borel $\sigma$-algebra $\cB(\cH)$.
Consider a mapping $A:E\times\cH \to 2^\cH$ and define
for any $\lambda>0$, the resolvent and the Yosida approximation
of $A$ as the mappings $J_\lambda$ and $A_\lambda$ respectively defined on $E\times\cH \to 2^\cH$ by
\begin{align*}
  J_\lambda(s,x) &= (I+\lambda A(s,\,.\,))^{-1}(x) \\
  A_\lambda(s,x) &= (x-J_\lambda(s,x))/\lambda
\end{align*}
for all $(s,x)\in E\times \cH$. 

\begin{assum}
\label{hyp:A}
\hfill 
\begin{enumerate}[(i)]
\item For every $s\in E$ $\mu$-a.e., $A(s,\,.\,)$ is maximal monotone.  
\item For any $\lambda>0$ and $x\in \cH$, $J_\lambda(\,.\,,x)$ is $\cE/\cB(\cH)$-measurable.
\end{enumerate}
\end{assum}
\pb{By~\cite[Lemme 2.1]{attouch1979familles}, the second point is equivalent
to the assumption that $A$ is $\cE\otimes \cB(\cH)$-Effros measurable. 
Also, by the same result, the statement ``for any $\lambda>0$'' in Assumption \ref{hyp:A}(ii) can be equivalently replaced
by ``there exists $\lambda>0$''.}
 As $A(s,\,.\,)$ is maximal monotone, $J_\lambda(s,\,.\,)$ is a
  single-valued continuous map for each $s\in E$. Thus, $J_\lambda$ is a
  Carathéodory map.  As such,
  $J_\lambda$ is  $\cE\otimes \cB(\cH) /\cB(\cH)$-measurable by \cite[Lemma 8.2.6]{aubin2009set}.

Consider an other probability space $(\Omega,\cF,\bP)$ and let $(\xi_n:n\in \bN^*)$ be a sequence of random variables
on $\Omega\to E$. For an arbitrary initial point $x_0\in \cH$ (assumed fixed throughout the paper), we consider the following iterations
\begin{equation}
  \label{eq:algo}
  x_{n+1} = J_{\lambda_{n}}(\xi_{n+1},x_n)\,.
\end{equation}

\begin{assum}
\label{hyp:step-iid}
\hfill 
\begin{enumerate}[(i)]
\item The sequence $(\lambda_n:n\in \bN)$ is positive and belongs
  to $\ell^2\backslash \ell^1$.
\item The random sequence $(\xi_n:n\in \bN^*)$ is independent and identically distributed with probability distribution $\mu$.
\end{enumerate}
\end{assum}
Let $\cF_n$ be the $\sigma$-algebra generated by the r.v. $\xi_1,\dots,\xi_n$.
We denote by $\bE$ the expectation on $(\Omega,\cF,\bP)$ and by
$\bE_n = \bE(\,.\,|\cF_n)$ the conditional expectation w.r.t. $\cF_n$. 


\subsection{Mean operator}

For any $x\in \cH$, we define $S_A(x) = \cS(A(\,.\,,x))$ as the set of measurable selections of $A(\,.\,,x)$.
We define similarly $S_A^p(x) = \cS^p(A(\,.\,,x))$.
For each $s\in E$, we set
$
D_s = \dom(A(s,\,.\,)).
$
Following \cite{hiriart1977contributions}, we define the 
essential intersection (or continuous intersection) of the domains $D_s$ as
$$
\cD = \bigcup_{N\in \mathscr{N}} \bigcap_{s\in E\backslash N} D_s
$$
where $\mathscr{N}$ is the set of $\mu$-negligible subsets of $E$.
Otherwise stated, a point $x$ belongs to $\cD$ if $x\in D_s$ for every
$s$ outside a negligible set. We define
$$
\underline A(x) = \int A(s,x)d\mu(s)\,.
$$
For any $s\in E$ and any $x\in D_s$, we define $A_0(s,x) = \mathrm{proj}_{A(s,x)}(0)$ as the element of least norm in $A(s,x)$.
\begin{lemma}
  \label{lem:Aumann}
Under Assumption~\ref{hyp:A},  $\underline A$ is monotone and has convex values.
Moreover, if $\int \|A_0(s,x)\|d\mu(s)<\infty$ for all $x\in\cD$, then $$\dom(\underline A)=\cD\,.$$  
\end{lemma}
\begin{proof}
The first point is clear.
For any $x\in\cD$, $A_0(\,.\,,x)$ is well defined $\mu$-a.e. and is measurable as the pointwise limit of measurable functions
$A_\lambda(\,.\,,x)$ for $\lambda\downarrow 0$.
By the measurable selection theorem, $\cD=\dom(S_A)$.
On the other hand, $\dom(\underline A) = \dom(S_A^1)\subset \cD$.
For any $x\in \cD$, $A_0(\,.\,,x)$ is an integrable selection of $A(\,.\,,x)$ by the standing hypothesis. 
Thus, $x\in \dom(\underline A)$. As a consequence, $\cD\subset \dom(\underline A)$.
\end{proof}

\begin{example}
\label{ex:fini}
  Consider the case where $\mu$ is a finitely supported measure, say $\mathrm{supp}(\mu)=\{1,\dots,m\}$
for some integer $m\geq 1$. Set $w_i=\mu(\{i\})$ for each $i$. 
Then $\underline A = \sum_{i=1}^m w_i A(i,\,.\,)$ and its domain is equal to
$$\cD = \bigcap_{i=1}^m D_i\,.$$
Moreover, if the  interiors of the respective sets $D_1,\dots, D_m$
have a non-empty intersection, then $\underline A$ is maximal by~\cite{rockafellar1970maximality}.
\end{example}


\begin{example}
  Set $\cH=\bR^d$. Assume $A$ is non-empty valued and for all $x\in \cH$,
$|A(\,.\,,x)|\leq g(.)$ for some $g\in L^1(E,\bR,\mu)$.
Then $\underline A$ is non-empty (convex) valued and has a closed graph by \cite{yannelis1990upper}.
Thus $\underline A$ is maximal monotone by \cite[pp. 45]{barbu2010nonlinear}.
\end{example}

\pb{ 
\begin{example}
Let $f:E\times\cH\to(-\infty,+\infty]$ be a normal convex integrand and assume that its integral functional
$F$ given by~(\ref{eq:F}) is proper. Then $F$ is convex and lower semicontinuous \cite{walkup1967stochastic}.
Let $A(s,x) = \partial f(s,x)$. Assume that the interchange between expectation and subdifferential operators holds \emph{i.e.},
$$
\int \partial f(s,x)d\mu(s) = \partial \int \! f(s,x)d\mu(s)\,,
$$
otherwise stated, $\underline A(x) = \partial F(x)$. Then, as $F$ is proper convex and lower semicontinuous,
it follows that $\underline A$ is maximal monotone  \cite[Theorem 21.2]{bauschke2011convex}. 
Sufficient conditions for the interchange can be found in \cite{rockafellar1982interchange}.
Assume that $F(x)<+\infty$ for every $x$ such that $x\in \dom f(s,\,.\,)$ $\mu$-almost everywhere.
Suppose that $F$ is continuous at some point 
and that the set valued function $s\mapsto \mathrm{cl} (\dom f(s,\,.\,))$ 
is constant almost everywhere. Then the identity $\underline A(x) = \partial F(x)$ holds.
\end{example}
}

We denote by $\zer(\underline A)=\{x\in \cH :0\in \underline A(x)\}$ the set of zeroes of $\underline A$. We define
for each $p\geq 1$
$$
\cZ_A(p) = \{x\in \cH\,:\,\exists\,\phi\in S_A^p(x)\,:\,\textstyle{\int \phi\,d\mu}=0\}\,. 
$$
For any $p\geq 1$,  $\cZ_A(p)\subset \cZ_A(1)$ and $\cZ_A(1)=\zer(\underline A)$. 

\pb{
\subsection{Outline of the proofs}
Before going into the details, we first provide an informal overview of the proof structure without insisting on the hypotheses for the moment.

We start by showing two separate results in Sections~\ref{sec:stability-1} and~\ref{sec:stability-2} respectively, which we merge in Section~\ref{sec:stability-3}.
The first result (Proposition~\ref{prop:stability}) states that almost surely, $\lim_{n\to\infty} \|x_n-x^\star\|$ exists for every $x^\star\in \cZ_A(2)$.
In particular, sequence $(x_n)$ is bounded with probability one, whenever $\cZ_A(2)$ is non-empty.
The second result (Theorem~\ref{th:CVcond}) states the following: when $\underline A$ is maximal, all weak cluster points of 
the averaged sequence $(\overline x_n)$ are zeroes of $\underline A$, almost surely on the event
\begin{equation}
\left\{ \omega\,:\,n\mapsto \frac{\sum_{k\leq n}\|A_{\lambda_k}(\,.\,,x_k(\omega))\|}{\sum_{k\leq n}\lambda_k}\text{ is uniformly integrable}\right\}\,.
\label{eq:eventUI}
\end{equation}
Assuming that $\zer(\underline A)\subset \cZ_A(2)$, the above results can be put together by straightforward application 
of Opial's lemma (see Lemma~\ref{lem:opial}).
Almost surely on the event~(\ref{eq:eventUI}), $(\overline x_n)$ converges weakly to a point in $\zer(\underline A)$.
The latter result is stated in Theorem~\ref{th:opial}. 
In order to complete the convergence proof, the aim is therefore to provide verifiable conditions under which the event~(\ref{eq:eventUI})
is realized almost surely. This point is addressed in Section~\ref{sec:cv}.

Checking that~(\ref{eq:eventUI}) holds w.p.1 is relatively easy in the special case where the domains $D_s$ are all equal to the same set $\cD$.
Using the inequality $\|A_{\lambda_k}(\,.\,,x_k)\|\leq \|A_{0}(\,.\,,x_k)\|$ and assuming that for every bounded set $K$, the family of measurable functions 
$(\|A_0(\,.\,,x)\|)_{x\in K\cap\cD}$ is uniformly integrable, the result follows (see Corollary~\ref{coro:CVdom}).
On the other hand, when the domains $D_s$ are not equal to the same set $\cD$, more developments are needed to prove that the event~(\ref{eq:eventUI}) is indeed realized w.p.1.
This point is addressed in Section~\ref{sec:distinct-domains} and the main result of the paper is eventually provided in Theorem~\ref{the:main}.
As opposed to the case of identical domains, the difficulty comes from the fact that the inequality 
$\|A_{\lambda_k}(\,.\,,x_k)\|\leq \|A_{0}(\,.\,,x_k)\|$ holds only if $x_k\in \cD$, which has no reason to be satisfied in the case of different domains. 
Instead, a solution is to pick some $z_k\in\cD$ close enough to $x_k$ in the sense that $\|z_k-x_k\|\leq 2d(x_k,\cD)$. 
Using that $A_\lambda(s,\,.\,)$ is $1/\lambda$-lipschitz continuous for every $s$, one has
\begin{equation}
\|A_{\lambda_k}(s,x_k)\|\leq \|A_{\lambda_k}(s,z_k)\| + \frac{2d(x_k,\cD)}{\lambda_k}\,.\label{eq:separationAlambda}
\end{equation}
As $z_k\in \cD$, the inequality $\|A_{\lambda_k}(\,.\,,z_k)\|\leq \|A_{0}(\,.\,,z_k)\|$ can be used and the first term
in the righthand side of~(\ref{eq:separationAlambda}) can be handled similarly to the previous case where the 
domains $D_s$ were assumed identical. In order to establish that (\ref{eq:eventUI}) is realized w.p.1, the remaining task 
is therefore to provide an estimate of the second term $\frac{d(x_k,\cD)}{\lambda_k}$. The latter estimate is provided
in Proposition~\ref{prop:dsum} which deeply relies on the mathematical developments of Lemma~\ref{lem:proj-perturb}.

In Section~\ref{sec:appli}, we particularize the algorithm to the case of convex programming. The proofs
of the section mainly consist in checking the conditions of application of the results of Section~\ref{sec:cv}.
}

\section{Stability and cluster points}
\label{sec:stability}

The following simple Lemma will be used twice.
\begin{lemma}
Let Assumption~\ref{hyp:A} hold true. Consider $u\in \cH$, $\phi\in S^1_A(u)$, $x\in \cH$, $\lambda>0$, $\beta>0$. Then,
for every $s$ $\mu$-a.e.,
\begin{equation}
  \ps{A_\lambda(s,x)-\phi(s),x-u} \geq  \lambda (1-\beta)  \|A_{\lambda}(s,x)\|^2
-\frac{\lambda}{4\beta}\|\phi(s)\|^2\,.
\label{eq:ps}
\end{equation}
\label{lem:ps}
\end{lemma}
\begin{proof}
As $\ps{A_{\lambda}(s,x)-\phi(s),J_\lambda(s,x)-u}\geq 0$ for all $s$ $\mu$-a.e., we obtain
\begin{align*}
  \ps{A_\lambda(s,x)-\phi(s),x-u} &\geq  \ps{A_{\lambda}(s,x)-\phi(s),x-J_\lambda(s,x)} \\
 &=  \lambda  \ps{A_{\lambda}(s,x)-\phi(s),A_\lambda(s,x)}\\
& =   \lambda \|A_{\lambda}(s,x)\|^2-\lambda \ps{\phi(s),A_\lambda(s,x)}\,.
\end{align*}
Use $\ps{a,b}\leq \beta\|a\|^2+\frac 1{4\beta}\|b\|^2$ with $a=A_\lambda(s,x)$ and $b=\phi(s)$, the result is proved.
\end{proof}

\subsection{Boundedness} 
\label{sec:stability-1}
The following proposition establishes that the stochastic proximal point algorithm is stable whenever $\cZ_A(2)$ is non-empty.

\begin{prop}
 Let Assumptions~\ref{hyp:A},~\ref{hyp:step-iid} hold true.
Suppose $\cZ_A(2)\neq\emptyset$ and let $(x_n)$ be defined by~(\ref{eq:algo}). Then,
\begin{enumerate}[(i)]
\item There exists an event $B\in\cF$ such that $\bP(B)=1$ and for every $\omega\in B$ and every $x^\star\in \cZ_A(2)$,
the sequence $(\|x_n(\omega)-x^\star\|)$ converges as $n\to\infty$.
\item $\bE(\sum_n\lambda_n^2\int\|A_{\lambda_n}(s,x_n)\|^2d\mu(s))<\infty$,
\item For any $p\in \bN^*$ such that $\cZ_A(2p)\neq\emptyset$, $\sup_n \bE(\|x_n\|^{2p})<\infty.$
\end{enumerate}
\label{prop:stability}
\end{prop}
\begin{proof}
Consider $u\in \cZ_A(2)$, $\phi\in S_A^2(u)$ such that $\int\phi d\mu=0$. 
Choose $0<\beta\leq\frac 12$. Note that $x_{n+1} = x_n-\lambda_n A_{\lambda_n}(\xi_{n+1},x_n)$. We expand
\begin{align*}
  \|x_{n+1}-u\|^2 &= \|x_n-u\|^2 +2\lambda_n  \ps{x_{n+1}-x_n,x_n-u}+ \lambda_n^2\|x_{n+1}-x_n\|^2 \\
&=\|x_n-u\|^2 -2\lambda_n  \ps{A_{\lambda_n}(\xi_{n+1},x_n),x_n-u}+ \lambda_n^2\|A_{\lambda_n}(\xi_{n+1},x_n)\|^2 \,.
\end{align*}
Using Lemma~\ref{lem:ps}, for all $s$ $\mu$-a.e.,
$$
\ps{A_{\lambda_n}(s,x_n),x_n-u}\geq \lambda_n (1-\beta)  \|A_{\lambda_n}(s,x)\|^2-\frac{\lambda_n}{4\beta}\|\phi(s)\|^2+ \ps{\phi(s),x_n-u}\,.
$$
Therefore,
\begin{multline}
\label{eq:dev-mse} \|x_{n+1}-u\|^2 \leq \|x_n-u\|^2 -\lambda_n^2(1-2\beta)  \|A_{\lambda_n}(\xi_{n+1},x)\|^2 \\
+\frac{\lambda_n^2}{2\beta}\|\phi(\xi_{n+1})\|^2 -2\lambda_n\ps{\phi(\xi_{n+1}),x_n-u}\,.
\end{multline}
Take the conditional expectation of both sides of the inequality:
$$
\bE_n \|x_{n+1}-u\|^2 \leq \|x_n-u\|^2 -\lambda_n^2(1-2\beta)  \int \|A_{\lambda_n}(s,x)\|^2d\mu(s)+\frac{\lambda_n^2c}{2\beta}
$$
where we set $c=\int\|\phi\|^2d\mu$ and used $\int \phi d\mu=0$.
By the Robbins-Siegmund theorem (see \cite[Theorem 1]{robbins1971convergence}) and choosing $0< \beta<\frac 12$, we deduce that:
$$\sum \lambda_n^2 \int \|A_{\lambda_n}(\,.\,,x_n)\|^2d\mu<\infty$$ (thus, point (ii) is proved), 
$\sup_n\bE(\|x_n\|^2)<\infty$ and finally, the 
sequence $(\|x_n-u\|^2)$ converges almost surely as $n\to\infty$. Let $Q$ be a dense countable subset of $\cZ_A(2)$.
There exists $B\in \cF$ such that $\bP(B)=1$ and for all $\omega\in B$, all $u\in Q$,  $(\|x_n(\omega)-u\|)$ converges.
Consider $\omega\in B$ and $x^\star\in \cZ_A(2)$. For any $\epsilon>0$, choose $u\in Q$ such that $\|x^\star-u\|\leq \epsilon$
and define $\ell_u=\lim_{n\to\infty}\|x_n(\omega)-u\|$.
Note that $\|x_n(\omega)-u\|\leq \|x_n(\omega)-x^\star\|+\epsilon$ thus $\ell_u\leq \liminf\|x_n(\omega)-x^\star\|+\epsilon$.
Similarly, $\|x_n(\omega)-x^\star\|\leq \|x_n(\omega)-u\|+\epsilon$ thus $\limsup\|x_n(\omega)-x^\star\|\leq \ell_u+\epsilon$.
Finally, $\limsup\|x_n(\omega)-x^\star\|\leq \liminf\|x_n(\omega)-x^\star\|+2\epsilon$.
As $\epsilon$ is arbitrary, we conclude that $(\|x_n(\omega)-x^\star\|)$ converges. Point (i) is proved.

We prove point (iii) by induction. Set $u\in \cZ_A(2p)$. We have shown above that $\sup_n\bE(\|x_n-u\|^2)<\infty$.
Consider an integer $q\leq p$ such that $\sup_n\bE(\|x_n-u\|^{2q-2})<\infty$. 
We will show that $\sup_n\bE(\|x_n-u\|^{2q})<\infty$ and the proof will be complete.
Use Equation~(\ref{eq:dev-mse}) with $\beta=\frac 12$,
\begin{align}
  \bE[\|x_{n+1}-u\|^{2q}] &\leq \bE[ (\|x_n-u\|^2 +\lambda_n^2\|\phi(\xi_{n+1})\|^2 -2\lambda_n\ps{\phi(\xi_{n+1}),x_n-u})^q] \nonumber \\
  & =\sum_{k_1+k_2+k_3=q}{q \choose k_1,k_2,k_3} T_n^{(k_1,k_2,k_3)}
\label{eq:multinomial}
\end{align}
where for any $\vec k=(k_1,k_2,k_3)$ such that $k_1+k_2+k_3=q$, we define
$$
 T_n^{\vec k} =(-2)^{k_3}\lambda_n^{2k_2+k_3}\bE[\|x_n-u\|^{2k_1}\|\phi(\xi_{n+1})\|^{2k_2} \ps{\phi(\xi_{n+1}),x_n-u}^{k_3}] \,.
$$
Note that $T_n^{(q,0,0)}=\bE[\|x_n-u\|^{2q}]$. We now prove that there exists a constant $c''$ such that for any 
$\vec k\neq (q,0,0)$, $|T_n^{\vec k}|\leq c''\lambda_n^2$. Consider a fixed value of $\vec k\neq (q,0,0)$ such that $k_1+k_2+k_3=q$
and consider the following cases. \\
$\bullet$ {If $k_3 = 0$}, then $k_1\leq q-1$ and $k_2\geq 1$. In that case, 
  \begin{align*}
    |T_n^{\vec k}| &\leq \lambda_n^{2k_2}\bE(\|x_n-u\|^{2k_1}){\textstyle \int\|\phi\|^{2k_2}d\mu} \\
 &\leq \alpha \lambda_n^{2}\bE(1+\|x_n-u\|^{2q-2}){\textstyle \int\|\phi\|^{2p}d\mu} 
  \end{align*}
where $\alpha$ is a constant chosen in such a way that $\lambda_n^{2k_2}\leq \alpha \lambda_n^2$ for any $1\leq k_2\leq q$
and where we used the inequality $a^{k_1}\leq 1+a^{q-1}$ for any $k_1\leq q-1$.
The constant $c'=\alpha \sup_n\bE(1+\|x_n-u\|^{2q-2}){\textstyle \int\|\phi\|^{2p}d\mu}$ is finite and we have
$|T_n^{\vec k}|\leq c'\lambda_n^2$.\\
$\bullet$ {If $k_3=1$ and $k_2=0$}, then $T_n^{\vec k}=0$ using that $\int\phi d\mu=0$.\\
$\bullet$ {In all remaining cases},  $k_1\leq q-2$ and $k_2+k_3\geq 2$. By the Cauchy-Schwarz inequality,
  \begin{align*}
    |T_n^{\vec k}| &\leq 2^{k_3}\lambda_n^{2k_2+k_3}\bE[\|x_n-u\|^{2k_1+k_3}\|\phi(\xi_{n+1})\|^{2k_2+k_3} ]\\
&= 2^{k_3}\lambda_n^{2k_2+k_3}\bE[\|x_n-u\|^{2k_1+k_3} ]{\textstyle \int \|\phi\|^{2k_2+k_3}d\mu}\,.
  \end{align*}
Now $2k_2+k_3= k_2+q-k_1\leq 2p$ and $2k_1+k_3= k_1 + q-k_2\leq k_1+p\leq  2q-2$. Using again that 
$\sup_n\bE(1+\|x_n-u\|^{2q-2})<\infty$ and $\int\|\phi\|^{2p}d\mu<\infty$, we conclude that there
exists an other constant $c''\geq c'$ such that $|T_n^{\vec k}|\leq c''\lambda_n^2$.

We have shown that $|T_n^{(k_1,k_2,k_3)}|\leq c''\lambda_n^2$ whenever $k_1+k_2+k_3=q$ and $(k_1,k_2,k_3)\neq (q,0,0)$.
Bounding the rhs of~(\ref{eq:multinomial}), we obtain
$$
\bE[\|x_{n+1}-u\|^{2q}] \leq \bE[\|x_{n}-u\|^{2q}] + c'' \lambda_n^2
$$
which in turn implies that $\sup_n\bE[\|x_{n}-u\|^{2q}]<\infty$.
\end{proof}

\subsection{Weak cluster points}
\label{sec:stability-2}
For an arbitrary sequence $(a_n : n\in \bN)$, we use the notation $\overline a_n$ to represent the weighted averaged sequence
$ \overline a_n = \sum_{k=1}^n\lambda_ka_k/\sum_{k=1}^n\lambda_k\,.$ 

Recall that a family $(f_i:i\in I)$ of measurable functions on $E\to\bR_+$  is uniformly integrable if
$$
\lim_{a\to+\infty}\sup_i\int_{\{f_i>a\}}f_i\,d\mu = 0\,.
$$
\begin{definition}
  We say that a sequence $(u_n)\in\cH^{\bN^*}$ has the property $\overline{UI}$ if
the sequence 
$$
\qquad\frac{\sum_{k=1}^n\lambda_k\|A_{\lambda_k}(\,.\,,u_k)\|}{\sum_{k=1}^n\lambda_k}\qquad (n\in \bN^*)
$$
is uniformly integrable.
\end{definition}
\begin{assum}
\label{hyp:Abar}
The monotone operator $\underline A$ is maximal.
\end{assum}

\pb{Note that Assumption~\ref{hyp:Abar} is satisfied in Examples 1, 2 and 3 above.}
\begin{theo}
 Let Assumptions~\ref{hyp:A}--\ref{hyp:Abar} hold true and suppose that $\cZ_A(2)\neq\emptyset$.
Consider the random sequence $(x_n)$ given by~(\ref{eq:algo}) with weighted averaged sequence $(\overline x_n)$. 
Let $G\in \cF$ be an event such that for almost every $\omega\in G$,  $(x_n(\omega))$ has the property $\overline {UI}$.
Then, there exists $B\in \cF$ such that $\bP(B)=1$ and such that for every $\omega\in B\cap G$, all weak cluster
 points of the sequence $(\overline x_n(\omega))$ belong to $\zer(\underline A)$.
\label{th:CVcond}
\end{theo}
\begin{proof}
Denote $h_\lambda(x) = \int A_\lambda(s,x)d\mu(s)$ for any $\lambda>0$, $x\in \cH$.
We justify the fact that $h_\lambda(x)$ is well defined.
As $\underline A$ is maximal, its domain contains at least one point $u\in \cH$.
For such a point $u$, there exists $\phi\in S_A^1(u)$.
As $A_{\lambda}(s,\,.\,)$ is $\frac 1\lambda$-Lipschitz continuous, 
$\|A_{\lambda}(s,x)\|\leq \|A_{\lambda}(s,u)\| + \frac{1}{\lambda}\|x-u\|$.
Moreover $\|A_{\lambda}(s,u)\|\leq \|A_0(s,u)\|\leq \|\phi(s)\|$ and since
$\phi\in L^1(E,\cH,\mu)$, we obtain that  $A_{\lambda}(\,.\,,x)\in L^1(E,\cH,\mu)$.
This implies that $h_\lambda(x)$ is well defined for all $x\in \cH$, $\lambda>0$.
We write 
$$
x_{n+1} = x_n -\lambda_n h_{\lambda_n}(x_n) + \lambda_n\eta_{n+1}
$$
where $\eta_{n+1} = -A_{\lambda_n}(\xi_{n+1},x_n)+h_{\lambda_n}(x_n)$ is a $\cF_n$-adapted 
martingale increment sequence \emph{i.e.}, $\bE_n(\eta_{n+1})=0$.
Note that
$$
\bE_n \|\eta_{n+1}\|^2 \leq \int \|A_{\lambda_n}(s,x_n)\|^2 d\mu(s)
$$
and by Proposition~\ref{prop:stability}(ii), it holds that $\sum
\lambda_n^2 \bE_n \|\eta_{n+1}\|^2<\infty$ almost surely.  As a
consequence, the $\cF_n$-adapted martingale $\sum_{k\leq n}\lambda_k
\eta_{k+1}$ converges almost surely to a random variable which is
finite $\bP$-a.e. Along with Proposition~\ref{prop:stability}, 
this implies that there exists
an event $B\in \cE$ of probability one such that for any $\omega\in B\cap G$,
\begin{enumerate}[(i)]
\item  $(\sum_{k\leq n}\lambda_k \eta_{k+1}(\omega))$ converges,
\item  $(x_n(\omega))$ is bounded,
\item $\sum_n\lambda_n^2\int \|A_{\lambda_n}(\,.\,,x_n(\omega))\|^2d\mu$ is finite,
\item $(x_n(\omega))$ has the property $\overline{UI}$.
\end{enumerate}
From now on to the end of this proof, we fix such an $\omega$.
As it is fixed, we omit the dependency w.r.t. $\omega$ to keep notations
simple. We write for instance $x_n$ instead of $x_n(\omega)$
and what we refer to as constants can depend on $\omega$.

Let $(u,v)\in \graph(\underline A)$ and consider $\phi\in S_A^1(u)$ such that $v=\int\phi d\mu$. 
Denote by $\epsilon>0$ an arbitrary positive constant. 

We need some preliminaries. By (i), there exists an integer $N=N(\epsilon)$ such that  for all $n\geq N$, $\left\|\sum_{k=N}^n
 \lambda_{k} \eta_{k+1}\right\|\leq\epsilon$. Define $Y_n(s) = \|A_{\lambda_n}(s,x_n)\|$
and let $(\overline Y_n)$ represent the corresponding weighted averaged sequence.
As $(\overline Y_n)$ is uniformly integrable, the same holds for the sequence
$(\overline Y_n^{(N)})$ defined by
$$
\overline Y_n^{(N)} = \frac{\sum_{k=N}^n \lambda_k Y_k}{\sum_{k=N}^n \lambda_k}\,.
$$
In particular, there exists a constant $c$ such that
\begin{equation}
\sup_n \int \overline Y_n^{(N)} d\mu < c\,.\label{eq:int-YbarN}
\end{equation}
Moreover, by \cite[Proposition II-5-2]{neveu1964bases}, there exists $\kappa_\epsilon>0$ such that 
$$ 
\forall H\in \cE,\ \ \mu(H)<\kappa_\epsilon\ \Rightarrow\ \int_H \overline Y_n^{(N)}d\mu <\epsilon\,.
$$
Since $\mu(\{\|\phi\|>K\})\to 0$ as $K\to+\infty$, there exists $K_1$ (depending on $\epsilon$) such that for all
$K\geq K_1$, $\mu(\{\|\phi\|>K\})<\kappa_\epsilon$. For any such $K$, 
\begin{equation}
\int_{\{\|\phi\|>K\}} \overline Y_n^{(N)}d\mu <\epsilon\,.\label{eq:Keps}
\end{equation}
Denote $v_K=\int_{\{\|\phi\|>K\}} \phi d\mu$. Note that $v_K\to v$ by the dominated convergence theorem.
Thus, there exists $K_2$ such that for all $K\geq K_2$,  $\|v_K-v\|<\epsilon$. 
From now on, we set $K\geq \max(K_1,K_2)$. 

Using an idea from~\cite{andrieu2005stability}, we define a sequence
$(y_n:n\geq N)$ such that $y_N=x_N$ and $y_{n+1} = y_n -
\lambda_{n}h_{\lambda_{n}}(x_n)$ for all $n\geq N$.  By induction, $y_n
= x_n - \sum_{k=N}^{n-1}\lambda_{k} \eta_{k+1}$. In particular,
$\|y_n-x_n\|\leq \epsilon$. We expand
  \begin{align*}
    \|y_{n+1}-u\|^2 &= \|y_n-u\|^2 - 2 \lambda_{n}\ps{h_{\lambda_{n}}(x_n),y_n-u}+ \|y_{n+1}-y_n\|^2 \\
    &\leq \|y_n-u\|^2 - 2 \lambda_{n}\ps{h_{\lambda_{n}}(x_n),x_n-u}  + 2 \epsilon \lambda_{n}\|h_{\lambda_{n}}(x_n)\| 
    + \lambda_{n}^2\|h_{\lambda_{n}}(x_n)\|^2\,.
  \end{align*}
Define $\delta_{K,\lambda}(x) = \int_{\{\|\phi\|> K\}}A_{\lambda}(s,x)d\mu(s)$ and use Lemma \ref{lem:ps} with $\beta=1$:
\begin{align*}
  \ps{h_{\lambda_n}(x_n)-v_K,x_n-u} &\geq - \|\delta_{K,\lambda_n}(x_n)\| 
  \|x_n-u\| -\frac{\lambda_n K^2}{4} \\
&\geq - c \int_{\{\|\phi\|> K\}}Y_nd\mu   -\frac{\lambda_n K^2}{4} 
\end{align*}
where the constant $c$ is selected in such a way that $c>\sup_n \|x_n-u\|$. 
Using that  $\|v_K-v\|<\epsilon$,
$$
\ps{h_{\lambda_n}(x_n)-v,x_n-u}  \geq - c\epsilon   - c \int_{\{\|\phi\|> K\}}Y_nd\mu   -\frac{\lambda_n K^2}{4} \,.
$$ 
As a consequence,
  \begin{equation}
    \|y_{n+1}-u\|^2   \leq \|y_n-u\|^2 - 2 \lambda_{n}\ps{v,x_n-u}  +r_n
\label{eq:y}
  \end{equation}
where we define
\begin{align*}
r_n &= 2c\epsilon \lambda_n + \lambda_n^2 s_{n}
+ 2 \lambda_{n} c t_{n,K}
+ 2 \epsilon \lambda_{n}t_{n,0}\\
  s_{n} &= \|h_{\lambda_{n}}(x_n)\|^2 + K^2/2\\
  t_{n,a} &= \int_{\{\|\phi\|\geq a\}}Y_nd\mu \ \ \ \ (\forall a\in\{0,K\}).
\end{align*}
For any $a\in\{0,K\}$, denote
$$
\overline t_{n,a}^{(N)} = \frac{\sum_{k=N}^n \lambda_k t_{k,a}}{\sum_{k=N}^n \lambda_k}\,.
$$
By inequality (\ref{eq:int-YbarN}), $\overline t_{n,0}^{(N)}< c$.
By inequality~(\ref{eq:Keps}), $\overline t_{n,0}^{(N)}< \epsilon$.
By point (iii), $\sum_n\lambda_n^2\|h_{\lambda_n}(x_n)\|^2<\infty$.
Using Assumption~\ref{hyp:step-iid}(i), it follows that 
$$
\frac{\sum_{k=N}^n r_k}{\sum_{k=N}^n \lambda_k} < 6c\epsilon + o_n(1)
$$
where $o_n(1)$ stands for a sequence which converges to zero as $n\to\infty$.
Summing the inequalities~(\ref{eq:y}) down to rank $N$, and dividing by $2\sum_{k=N}^n\lambda_k$, we obtain
$$
0   \leq -  \frac{\sum_{k=N}^n \lambda_k\ps{v,x_k-u}}{\sum_{k=N}^n \lambda_k} +3 c \epsilon +o_n(1)\,.
$$
Let $\tilde x$ be a weak cluster point of the weighted averaged sequence $\overline x_n$.
Then, $\tilde x$ is also a weak cluster point of the sequence
$$
\frac{\sum_{k=N}^n \lambda_k x_k}{\sum_{k=N}^n \lambda_k} \,. 
$$
We obtain $0   \leq -  \ps{v,\tilde x-u} +3 c \epsilon$. The inequality holds for any $\epsilon>0$, thus
 $0\leq  -\ps{v,\tilde x-u}$. As the inequality holds for any 
$(u,v)\in \mathrm{gr}(\underline A)$ and $\underline A$ is maximal monotone, this means that
$(\tilde x,0)\in \mathrm{gr}(\underline A)$~\cite[Theorem 20.21]{bauschke2011convex}. 
\end{proof}

\subsection{Weak ergodic convergence}
\label{sec:stability-3}

The aim of Theorem~\ref{th:opial} below is to merge Proposition~\ref{prop:stability} and Theorem~\ref{th:CVcond} into a 
weak ergodic convergence result. We need the following condition to hold.

\begin{assum}
\label{hyp:zer}
  $\zer(\underline A)\neq\emptyset$ and $\zer(\underline A)\subset \cZ_A(2)$.
\end{assum}

The condition $\zer(\underline A)\neq\emptyset$ means that there exists $x^\star\in \cH$ for which one can find
a selection $\phi$ of $A(\,.\,,x^\star)$ such that  $\int\phi d\mu=0$.
The condition $\zer(\underline A)\subset \cZ_A(2)$ means that moreover, such a $\phi$ can be chosen
to be square integrable. For instance, this holds under the stronger condition that for any zero $x^\star$
of $\underline A$,  $|A(\,.\,,x^\star)|$ is square integrable.

\begin{lemma}[Passty]
  Let $(\lambda_n)$ be a non-summable sequence of positive reals, and $(a_n)$ any sequence in $\cH$
with weighted averaged $(\overline a_n)$. Assume there exists a non-empty closed convex subset $Q$ of $\cH$
such that (i) weak subsequential limits of $\overline a_n$ lie in $Q$ ; and (ii) $\lim_n \|a_n -b\|$ exists for
all $b\in Q$. Then $(\overline a_n)$ converges weakly to an element of $Q$.
\label{lem:opial}
\end{lemma}
\begin{proof}
  See \cite{passty1979ergodic}.
\end{proof}

\begin{theo}
\label{th:opial} 
   Let Assumptions~\ref{hyp:A}--\ref{hyp:zer} hold true.
Consider the random sequence $(x_n)$ given by~(\ref{eq:algo}) with weighted averaged sequence $(\overline x_n)$. 
Let $G\in \cF$ be an event such that for almost every $\omega\in G$,  $(x_n(\omega))$ has the property $\overline {UI}$.
Then, almost surely on $G$, $(\overline x_n)$ converges \pb{weakly} to a point in  $\zer(\underline A)$.
\end{theo}
\begin{proof}
It is a consequence of Proposition~\ref{prop:stability}(i), Theorem~\ref{th:CVcond} and Lemma~\ref{lem:opial}.
\end{proof}

Theorem~\ref{th:opial} establishes the almost sure weak ergodic convergence of the 
stochastic proximal point algorithm under the abstract condition that w.p.1, $(x_n)$ has the property $\overline{UI}$.
We must now provide verifiable conditions under this property
indeed holds w.p.1. This is the purpose of the next section.

\clearpage 
\section{Main results}
\label{sec:cv}

\subsection{Case of a common domain} We first address the case where the domains $D_s$ of the operators $A(s,\,.\,)$ ($s\in E$)
are equal (at least for all $s$ outside a neglible set). We also need an additional assumption.

\begin{assum}
\label{hyp:A0bound}
For any bounded set $K\subset\cH$, 
the family $(\|A_0(\,.\,,x)\|:x\in K\cap \cD)$ is uniformly integrable.
\end{assum}

Assumption~\ref{hyp:A0bound} is satisfied if the following stronger condition holds for any bounded set $K\subset\cH$:
\begin{equation}
\label{eq:moment}
\exists r_K>0,\ \sup_{x\in K\cap \cD} \int \|A_0(s,x)\|^{1+r_K}\ d\mu(s)<\infty\,.
\end{equation}

\begin{coro}
 Let Assumptions~\ref{hyp:A}--\ref{hyp:A0bound} hold true. 
Assume that the domains $D_s$ coincide for all $s$ outside a $\mu$-negligible set.
Consider the random sequence $(x_n)$ given by~(\ref{eq:algo}) with weighted averaged sequence $(\overline x_n)$. 
Then, almost surely, $(\overline x_n)$ converges weakly to a zero of $\underline A$. 
\label{coro:CVdom}
\end{coro}
\begin{proof}
By Proposition~\ref{prop:stability} and the fact that $D_s=\cD$ for all $s$ $\mu$-a.e.,
there is a set of probability one such that for any $\omega$ in that set, there is a bounded set
 $K=K_\omega$ such that $x_n(\omega)\in K\cap \cD$ for all $n\in \bN^*$. 
By Assumption~\ref{hyp:A0bound}, the sequence $(\|A_0(\,.\,,x_n(\omega))\|:n \in \bN^*)$ is uniformly integrable.
As  $\|A_{\lambda_n}(\,.\,,x_n(\omega))\|\leq \|A_0(\,.\,,x_n(\omega))\|$, the same holds for the sequence $(\|A_{\lambda_n}(\,.\,,x_n(\omega))\|:n\in \bN^*)$
and holds as well for the corresponding weighted averaged sequence.
The conclusion follows from Theorem~\ref{th:opial}.
\end{proof}

\subsection{Case of distinct domains}
\label{sec:distinct-domains}

\pb{We now address the case where the domains $D_s$ may vary with $s$.
The case is more involved, because the sole Assumption~\ref{hyp:A0bound} 
is not sufficient to ensure the convergence. The
  reason is that the inequality $\|A_{\lambda_n}(s,x_n)\|\leq
  \|A_0(s,x_n)\|$ used to prove Corollary~\ref{coro:CVdom} does no
  longer hold when $x_n\notin D_s$.  Nonetheless, using that
  $A_{\lambda}(s,\,.\,)$ is $\frac 1\lambda$-Lipschitz continuous, the
  argument can be adapted provided that the iterates converge ``quickly enough''
  to the essential domain~$\cD$. The crux of the paragraph is therefore
  to provide estimates of the distance between $x_n$ and the set $\cD$.
  To this end, we shall need some regularity conditions on the collection
  of sets $D_s$. These conditions can be seen as an extension to possibly infinitely many sets 
  of the \emph{bounded linear regularity} condition of Bauschke \emph{et al.} \cite{bauschke1999strong}.
}

We define the mapping $\Pi:E\times \cH\to\cH$  by 
$$
\Pi(s,x)=\mathrm{proj}_{\mathrm{cl}(D_s)}(x)\,.
$$
Note that $\Pi(s,x)=\lim_{\lambda\downarrow 0} J_\lambda(s,x)$ by \cite[Theorem 23.47]{bauschke2011convex}.
By Assumption~\ref{hyp:A}, $\Pi$ is $\cE\otimes\cB(\cH)/\cB(\cH)$-measurable as a pointwise limit of measurable maps.
The distance between a point $x\in \cH$ and $D_s$ coincides with $d(x,D_s) = \|x-\Pi(s,x)\|$.

\pb{
\begin{assum}
  \label{hyp:reg-domaine}
For every $M>0$, there exists $\kappa_M >0$ such that for all $x\in \cH$ such that $\|x\|\leq M$, 
$$\int d(x,D_s)^2d\mu(s)\geq \kappa_M \,d(x,\cD)^2\,.$$
\end{assum}
The above assumption is quite mild, and is easier to illustrate in the case of finitely many sets.
Following \cite{bauschke1999strong}, we say that a finite collection of closed convex subsets $(X_1,\dots,X_m)$ 
over some Euclidean space is \emph{boundedly linearly regular} if for every $M>0$, there exists $\kappa'_M>0$
such that for every $\|x\|\leq M$,
\begin{equation}
\label{eq:linearly-reg}
\max_{i=1\dots m} d(x,X_i)\geq \kappa'_M d(x,X)\ \text{ where }X=\bigcap_{i=1}^mX_i
\end{equation}
where implicitely $X\neq \emptyset$. Sufficient conditions for a collection of set can be found
in \cite{bauschke1999strong} and reference therein. For instance, the qualification condition
$\cap_i\mathrm{ri}(X_i)\neq\emptyset$ is sufficient to ensure that $X_1,\dots,X_m$ are boundedly linearly regular,
where $\mathrm{ri}$ stands for the relative interior.

Now consider the special case of Example~\ref{ex:fini} \emph{i.e.}, $\mu$ is finitely
supported. Assume that $\cH$ is a Euclidean space and that  the domains
$D_1,\dots,D_m$ of the operators $A(1,\,.\,),\dots,A(m,\,.\,)$ are closed.
It is routine to check that Assumption~\ref{hyp:reg-domaine} holds if and only if 
$D_1,\dots,D_m$ are boundedly linearly regular. 
}
\begin{lemma}
\label{lem:proj-perturb}
  Let Assumptions~\ref{hyp:A},~\ref{hyp:step-iid} and \ref{hyp:reg-domaine} hold true.
Assume that $\lambda_n/\lambda_{n+1}\to 1$  as $n\to +\infty$ and $\cD\neq \emptyset$.
For each $n$, consider a $\cF_n$-measurable random variable $\delta_n$ on $\cH$. Assume that the sequence $(\bE_n\|\delta_{n+1}\|^2)$
is bounded almost surely and in $L^1(\Omega,\cH,\bP)$. 
\pb{Consider the sequence $(x_n)$ given by
\begin{equation}
x_{n+1} = \Pi(\xi_{n+1},x_n) + \lambda_n \delta_{n+1}\,.\label{eq:proj-perturb}
\end{equation}
Assume that, with probability one, $(x_n)$ is bounded. Then 
\begin{align*} 
  & \sup_n \  \frac{\sum_{k\leq n} d(x_k,\cD)}{\sum_{k\leq n} \lambda_k}<\infty\,\, \text{a.s.}  
\end{align*}}
\end{lemma}

\begin{proof}
  Consider an arbitrary point $u\in \cD$. By definition of $\cD$,
  $u\in D_s$ for all $s$ $\mu$-a.e.  For any $\beta>0$,
  \begin{align*}
    \|x_{n+1}- u\|^2 &\leq  (1+\beta)\|\Pi(\xi_{n+1},x_n)-u\|^2+\lambda_n^2(1+\frac 1\beta)\|\delta_{n+1}\|^2\,.
  \end{align*}
As $\Pi(\xi_{n+1},\,.\,)$ is firmly non-expansive,
  \begin{equation*}
    \|x_{n+1}- u\|^2 \leq  (1+\beta)\left(\|x_n-u\|^2-\|x_n-\Pi(\xi_{n+1},x_n)\|^2\right)  +\lambda_n^2(1+\frac 1\beta)\|\delta_{n+1}\|^2\,.
  \end{equation*}
  The above inequality holds for any $u\in \cD$ and thus for any $u\in
  \mathrm{cl}(\cD)$. It holds in particular when substituting $u$ with
  $\mathrm{proj}_{\mathrm{cl}(\cD)}(x_n)$. Remarking that
  $d(x_{n+1},\cD)\leq \|x_{n+1}-
  \mathrm{proj}_{\mathrm{cl}(\cD)}(x_n)\|$, it follows that
  \begin{equation*}
    d(x_{n+1},\cD)^2 \leq (1+\beta)\left(d(x_{n},\cD)^2-\|x_n-\Pi(\xi_{n+1},x_n)\|^2\right) 
    +\lambda_n^2(1+\frac 1\beta)\|\delta_{n+1}\|^2\,.
  \end{equation*}
\pb{Consider a fixed $M>0$, and denote by $B_n^M$ the probability event $\cap_{k\leq n}\{\|x_k\|\leq M\}$.
Denote by $\chi_B$ the characteristic function of a set $B$, equal to 1 on $B$ and to zero outside.
By Assumption~\ref{hyp:reg-domaine},
\begin{align*}
  \bE_n(\|x_n-\Pi(\xi_{n+1},x_n)\|^2\chi_{B_n^M}) &= \int
  \|x_n-\Pi(s,x_n)\|^2d\mu(s)\chi_{B_n^M}\\ &\geq \kappa_M\, d(x_n,\cD)^2\chi_{B_n^M}
\end{align*}
where $\kappa_M$ is the constant defined in Assumption~\ref{hyp:reg-domaine}.
Define $t_n=t_{n,M}$ as the random variable $t_n=d(x_n,\cD)^2\chi_{B_n^M}$.
Upon noting that $\chi_{B_{n+1}^M}\leq \chi_{B_{n}^M}$, we obtain
$$
   t_{n+1}^2 \leq (1+\beta)(1-\kappa_M)t_n^2
    +\lambda_n^2(1+\frac 1\beta)\|\delta_{n+1}\|^2\,.
$$
Taking the conditional expectation, $\bE_nt_{n+1}^2 \leq
(1+\beta)(1-\kappa_M)t_n^2 +\lambda_n^2(1+\frac
1\beta)\bE_n\|\delta_{n+1}\|^2$.  Define $\Delta_n = t_n/\lambda_n$. }
Using that $\lambda_n/\lambda_{n+1}\to 1$ and choosing $\beta$ small
enough, there exists constants $0<\rho<1$, $c>0$ and a deterministic
integer $n_0$ depending on the sequence $(\lambda_n)$ and the
constants $\beta$, $\kappa_M$ such that for all $n\geq n_0$,
  \begin{align}
    \bE_n(\Delta_{n+1}^2) &\leq \rho\,\Delta_n^2+ c\,\bE_n\|\delta_{n+1}\|^2\,.
\label{eq:lyapunov}
  \end{align}
Taking the expectation of both sides and using that $(\bE\|\delta_{n+1}\|^2)$ is bounded, we obtain
that the sequence $(\Delta_n)$ is uniformly bounded in $L^2(\Omega, \bR_+,\bP)$. Now consider the sums
$$
T_n = \sum_{k=n_0+1}^n \pb{t_k}\qquad \text{and} \qquad \varphi_n = \sum_{k=n_0+1}^n \lambda_k\,.
$$
Decompose $T_n = \sum_{k=n_0+1}^n \bE_{k-1}d(x_{k},\cD) + R_n$ where
$$
R_n = \sum_{k=n_0+1}^n (\pb{t_k}-\bE_{k-1}\pb{t_k})\,.
$$
Note that $R_n$ is an $\cF_n$-adapted martingale and 
\pb{$\bE((t_k-\bE_{k-1}t_k)^2)\leq \bE(t_k^2)\leq C \lambda_k^2$} for some finite constant $C = \sup_n\bE(\Delta_n^2)$.
As $\sum_k \lambda_k^2<\infty$, we deduce that $R_n$ converges a.s. to some r.v. $R_\infty$ which is finite $\bP$-a.e.
As a consequence, $R_n/\varphi_n$ tends a.s. to zero. On the other hand, by Jensen's inequality,
$$
T_n \leq  \sum_{k=n_0+1}^n \left(\bE_{k-1}\pb{t_k^2}\right)^{\frac 12} + \|R_n\|\,.
$$
By~(\ref{eq:lyapunov}) again and the assumption that $\bE_n\|\delta_{n+1}\|^2$ is bounded a.s., there exists
a finite r.v. $Z>0$ such that, almost surely, 
$\bE_n(\Delta_{n+1}^2) \leq \rho\,\Delta_n^2+ c\,Z$. Thus, there exists other constants $\rho<\rho_1<1$ and $c_1$ such that
$\bE_n(\Delta_{n+1}^2)^{1/2} \leq \rho_1\,\Delta_n+ c_1\,Z$. Using that $\lambda_n/\lambda_{n+1}\to 1$, we obtain
$$
\bE_n(\pb{t_{n+1}^2})^{1/2} \leq \rho_2\,\pb{t_{n+1}}+ c_1\lambda_{n+1}\,Z
$$
for some constants $\rho_1<\rho_2<1$. As a consequence,
$$
\frac{T_n}{\varphi_n} \leq \frac{c_2 Z}{1-\rho_2} + \frac{\|R_n\|}{(1-\rho_2)\varphi_n}\,.
$$
\pb{Therefore, for every $M>0$, the exist a probability one event on which $T_n/\varphi_n$ is bounded.
Hence, on a probability one set, for every integer $M>0$, the sequence
$$
 \frac{\sum_{k\leq n} d(x_k,\cD)\chi_{B_k^M}}{\sum_{k\leq n} \lambda_k}
$$
is bounded. As $(x_n)$ is bounded w.p.1., the conclusion follows.
}
\end{proof}


\begin{assum}
\label{hyp:Jreg} 
There exist $p\in\bN^*$ and $C\in L^2(\E,\bR_+,\mu)$ such that for any $x\in \cH$, $\lambda>0$,
$$
 {\|J_\lambda(s,x)-\Pi(s,x)\|}\leq \lambda\,C(s)(1+\|x\|^{p})
$$ 
and $\cZ_A(2p)\neq\emptyset$.
\end{assum}

\pb{We recall that  $J_\lambda(s,x)$ converges to the best approximation
$\Pi(s,x)$ of $x$ in $D_s$ when $\lambda\downarrow 0$. Assumption~\ref{hyp:Jreg}
provides an additional condition on the rate. Loosely speaking, the condition means that 
the resolvent value $J_\lambda(s,x)$ 
should be at distance $O(\lambda)$ from the projection $\Pi(s,x)$.
A sufficient condition will be provided in Section~\ref{sec:appli} in the case of subdifferentials.
}

The second condition $\cZ_A(2p)\neq\emptyset$ means that there exists a zero of $\underline A$, say $x^\star$, 
for which one can find a $(2p)$-integrable selection $\phi\in A(\,.\,,x^\star)$ such that $\int\phi d\mu=0$.
This is for instance the case if $|A(\,.\,,x^\star)|^{2p}$ is integrable.


\begin{prop}
    Let Assumptions~\ref{hyp:A},~\ref{hyp:step-iid},~\ref{hyp:reg-domaine} and~\ref{hyp:Jreg} hold true.
Suppose that $\lambda_n/\lambda_{n+1}\to 1$ as $n\to\infty$.
Then, the sequence $(x_n)$ given by~(\ref{eq:algo}) satisfies almost surely 
\begin{equation*}
\sup_n \  \frac{\sum_{k\leq n} d(x_{k},\cD)}{\sum_{k\leq n} \lambda_k}<\infty\,.  
\end{equation*}
\label{prop:dsum}
\end{prop}
\begin{proof} 
  The sequence $(x_n)$ satisfies (\ref{eq:proj-perturb}) if we set
$$\delta_{n+1} = (J_{\lambda_n}(\xi_{n+1},x_n)-\Pi(\xi_{n+1},x_n))/\lambda_n\,.$$
By Assumption~\ref{hyp:Jreg}, $\bE_n\|\delta_{n+1}\|^2\leq
c(1+\|x_n\|^{2p})$ for some constant $c>0$.  Therefore, by
Proposition~\ref{prop:stability}(iii), $\bE_n\|\delta_{n+1}\|^2$ is
uniformly bounded almost surely and in $L^1(\Omega,\cH,\bP)$. The
conclusion of Lemma~\ref{lem:proj-perturb} applies.
\end{proof}

\begin{theo}
\label{the:main}
  Let Assumptions \ref{hyp:A}--\ref{hyp:Jreg} hold true
and let $\lambda_n/\lambda_{n+1}\to 1$ as $n\to\infty$.
Consider the random sequence $(x_n)$ given by~(\ref{eq:algo}) with weighted averaged sequence $(\overline x_n)$. 
Then, almost surely, $(\overline x_n)$ converges weakly to a zero of  $\underline A$.
\end{theo}

\begin{proof}
For every $n$, choose any point $z_n\in \cD$  such that $\|z_n-x_n\|\leq 2 d(x_{n},\cD)$.
As $A_\lambda(s,\,.\,)$ is $\frac 1\lambda$-Lipschitz continuous, 
\begin{align*}
  \|A_{\lambda_n}(s,x_n)\| &\leq \|A_{\lambda_n}(s,z_n)\| + \frac  {2d(x_{n},\cD)}{\lambda_n}\,. 
\end{align*}
Using moreover that $\|A_{\lambda_n}(s,z_n)\|\leq \|A_0(s,z_n)\|$,
\begin{align*}
\frac{\sum_{k=1}^n\lambda_k\|A_{\lambda_k}(s,x_k)\|}{\sum_{k=1}^n\lambda_k}  &\leq \frac{\sum_{k=1}^n\lambda_k\|A_0(s,z_k)\|}{\sum_{k=1}^n\lambda_k} 
+  2\frac{\sum_{k=1}^nd(x_{k},\cD)}{\sum_{k=1}^n\lambda_k} \,.
\end{align*}
By Proposition~\ref{prop:dsum}, 
\begin{align}
\frac{\sum_{k=1}^n\lambda_k\|A_{\lambda_k}(s,x_k)\|}{\sum_{k=1}^n\lambda_k}  &\leq \frac{\sum_{k=1}^n\lambda_k\|A_0(s,z_k)\|}{\sum_{k=1}^n\lambda_k} +  C'
\label{eq:av-unif-int}
\end{align}
where $C'$ is a r.v. independent of $n$ and $s$ and which is finite $\bP$-a.e. 
By Assumption~\ref{hyp:A0bound}, the family $\|A_0(\,.\,,z_k(\omega))\|$ is uniformly integrable for almost every $\omega$.
Thus, the same holds for the corresponding averaged sequence, which in turn implies that
the functions of $s$ given by the lhs of~(\ref{eq:av-unif-int}) are uniformly integrable.
The conclusion follows from Theorem~\ref{th:CVcond}.
\end{proof}

\pb{
\subsection{Strong monotonicity and strong convergence}

We prove the following.

\begin{theo}
Let Assumptions~\ref{hyp:A},~\ref{hyp:step-iid} hold true.
Assume that for every $s\in E$, $A(s,\,.\,)$ is strongly monotone with modulus $\alpha(s)$ where
$\alpha:E\to\bR_+$ is a measurable function such that $\bP(\alpha(\xi_1)\neq 0)>0$.
Then $\underline A$ is strongly monotone and, as such, admits a unique zero $x^\star$. If $x^\star\in \cZ_A(2)$ then, almost surely, the sequence $(x_n)$ defined by~(\ref{eq:algo})
converges strongly to $x^\star$.
\end{theo}

\begin{proof}
  Set $(x,y)$ and $(x',y')$ in $\graph(\underline A)$. Let $\phi$ and $\phi'$ be integrable selections of $A(\,.\,,x)$ and $A(\,.\,,x')$ respectively
such that $y=\int\phi d\mu$ and $y'=\int\phi' d\mu$. Then,
$$
\ps{\phi(s)-\phi'(s),x-x'}\geq \alpha(s) \|x-x'\|^2\,.
$$
Integrating over $s$ and noting that $\int \alpha d\mu>0$ by hypothesis, we deduce that $\underline A$ is strongly monotone.
Let $x^\star$ be its unique zero and assume that $x^\star\in \cZ_A(2)$. Note that there is no restriction in assuming that $\alpha(\,.\,)\leq 1$
(otherwise just replace $\alpha(\,.\,)$ with $\min (\alpha(\,.\,),1)$).

By strong monotonicity, 
the inequality~(\ref{eq:ps}) of Lemma~\ref{lem:ps} can be replaced by
\begin{equation*}
  \ps{A_\lambda(s,x)-\phi(s),x-u} \geq  \alpha(s)\| J_\lambda(s,x)-u\|^2+\lambda (1-\beta)  \|A_{\lambda}(s,x)\|^2
-\frac{\lambda}{4\beta}\|\phi(s)\|^2\,.
\end{equation*}
As a consequence, Equation~(\ref{eq:dev-mse}) can be replaced by
\begin{multline}
\label{eq:dev-mse-strg}
\|x_{n+1}-x^\star\|^2 \leq \|x_n-x^\star\|^2 -\lambda_n^2(1-2\beta)  \|A_{\lambda_n}(\xi_{n+1},x_n)\|^2 \\
-2\lambda_n\alpha(\xi_{n+1})\|x_{n+1}-x^\star\|^2+\frac{\lambda_n^2}{2\beta}\|\phi(\xi_{n+1})\|^2 -2\lambda_n\ps{\phi(\xi_{n+1}),x_n-x^\star}
\end{multline}
where $\phi$ is a measurable selection of $A(\,.\,,x^\star)$ such that $\int \phi d\mu=0$.
In the sequel, we shall simply set $\beta=\frac 12$. On the other hand, by straightforward algebra,
\begin{align*}
  \|x_{n+1}-x^\star\|^2 &\geq \|x_n-x^\star\|^2 + 2\ps{x_{n+1}-x_n,x_n-x^\star} \\
&= \|x_n-x^\star\|^2 - 2\lambda_n \ps{A_{\lambda_n}(\xi_{n+1},x_n),x_n-x^\star} \\
&\geq \|x_n-x^\star\|^2 - \lambda_n \|A_{\lambda_n}(\xi_{n+1},x_n)\|^2-\lambda_n\|x_n-x^\star\|^2
\end{align*}
and by plugging the above inequality into~(\ref{eq:dev-mse-strg}), using $\alpha(\,.\,)\leq 1$ and recalling $\beta=\frac 12$,
\begin{multline*}
\|x_{n+1}-x^\star\|^2 \leq (1+2\lambda_n^2)\|x_n-x^\star\|^2
-2\lambda_n\alpha(\xi_{n+1})\|x_{n}-x^\star\|^2
+2\lambda_n^2\|A_{\lambda_n}(\xi_{n+1},x_n)\|^2\\
+\lambda_n^2\|\phi(\xi_{n+1})\|^2 -2\lambda_n\ps{\phi(\xi_{n+1}),x_n-x^\star}\,.
\end{multline*}
Applying the conditional expectation $\bE_n$ on both sides, and setting $\bar \alpha = \int\alpha d\mu$, 
and $V_n = 2\E_n\|A_{\lambda_n}(\xi_{n+1},x_n)\|^2 + \int \|\phi\|^2d\mu$, we obtain
\begin{equation*}
\E_n(\|x_{n+1}-x^\star\|^2) \leq (1+2\lambda_n^2)\|x_n-x^\star\|^2
-2\lambda_n\bar \alpha \|x_{n}-x^\star\|^2 + \lambda_n^2 V_n\,.
\end{equation*}
By Proposition~\ref{prop:stability}(ii) and the fact that $(\lambda_n)\in \ell^2$, one has $\sum_n\lambda_n^2V_n<\infty$ a.s.
Therefore, by \cite{robbins1971convergence}, 
$$
\sum_n\lambda_n\bar \alpha \|x_{n}-x^\star\|^2 <\infty\ \text{a.s.}
$$
By the standing hypothesis, $\bar\alpha>0$, thus $\sum_n\lambda_n \|x_{n}-x^\star\|^2 <\infty$ a.s.
Since $\|x_{n}-x^\star\|$ converges a.s. by Proposition~\ref{prop:stability}(i) and since $(\lambda_n)\notin \ell^1$, it follows that $\|x_{n}-x^\star\|\to 0$.

\end{proof}

}

\section{Application to convex programming}
\label{sec:appli}

\subsection{Problem and Algorithm}


\pb{Consider the context of Example 3. 
Let $f:E\times \cH\to (-\infty,+\infty]$ be a normal convex integrand.
Denote by $F(x)=\int f(s,x)d\mu(x)$ the corresponding integral functional.
}
Identifying $\partial f$ with the operator $A$ of Section~\ref{sec:algo}, 
the resolvent $J_{\lambda}$ coincides with the proximity operator $(s,x)\mapsto \mathrm{prox}_{\lambda f(s,\,.\,)}(x)$ 
defined in~(\ref{eq:proximity-op}). The iterations~(\ref{eq:algo}) write
\begin{equation}
x_{n+1} = \mathrm{prox}_{\lambda_n f(\xi_{n+1},\,.\,)}(x_n)\,.
\label{eq:algo-fct}
\end{equation}
The aim is to prove the almost sure weak convergence in average of $(x_n)$ to a minimizer of $F$ (assumed to exist).
We denote by $\partial f_0(s,x)$ the element of $\partial f(s,x)$ with smallest norm.
We denote by $\cD$ the essential intersection of the sets  $D_s=\dom(\partial f(s,\,.\,))$ for $s\in E$.

\begin{assum}
\hfill 
\begin{enumerate}[(i)]
\item $f:E\times \cH\to (-\infty,+\infty]$ is a normal convex integrand.
\item $F$ is proper and lower semicontinuous.
\item For all $x\in \cH$, $\partial F(x) = \int \partial f(s,x)d\mu(s)$.
\item The set of minimizers of $F$ is non-empty and included in $\cZ_{\partial f}(2)$.
\end{enumerate} 
\label{hyp:integrand}
\end{assum}

\smallskip

Assumption~\ref{hyp:integrand}(iii) has been discussed in Example 3.

\subsection{Case of a common domain}\hspace*{\fill} 
\label{sec:CVdom-fct}

\begin{theo}
 Let Assumptions~\ref{hyp:step-iid} and \ref{hyp:integrand} hold true.
Assume that the domains $D_s$ coincide for all $s$ outside a $\mu$-negligible set.
Assume that for any bounded set $K\subset \cH$, the family $(\|\partial f_0(\,.\,,x)\|:x\in K\cap \cD)$ is uniformly integrable.
Consider the random sequence $(x_n)$ given by~(\ref{eq:algo-fct}) with weighted averaged sequence $(\overline x_n)$. 
Then, almost surely,  $(\overline x_n)$ converges weakly to a minimizer of $F$. 
\label{the:CVdom-fct}
\end{theo}

\begin{proof}
  We prove that $A=\partial f$ satisfies the conditions of
  Assumptions~\ref{hyp:A} and \ref{hyp:Abar} and the conclusion
  follows from Corollary~\ref{coro:CVdom}.  Operator $\partial
  f(s,\,.\,)$ is maximal monotone for any given $s\in E$, see
  \emph{e.g.} \cite[Theorem 21.2]{bauschke2011convex}.  
  For a fixed $x\in \cH$, $\partial f(\,.\,,x)$ is measurable, see
  \cite[Corollary 4.6]{rockafellar1969measurable} and \cite[Theorem
  3]{papageorgiou1997convex} in the infinite dimensional case.  The
  proximity operator $J_\lambda(\,.\,,x)$ is $\cE/\cB(\cH)$
  measurable, see \cite[Lemma 4]{rockafellar1968integrals} (combined
  with \cite[Proposition 2]{rockafellar1971convex} in the infinite
  dimensional case).  Therefore, $A=\partial f$ satisfies the
  conditions in Assumption~\ref{hyp:A}.

  Note that $F$ is a convex function.  By Assumption~\ref{hyp:integrand}(ii)
and \cite[Theorem 21.2]{bauschke2011convex}, $\partial F$ is maximal monotone.
Using moreover Assumption~\ref{hyp:integrand}(iii), the condition in
  Assumption~\ref{hyp:Abar} is satisfied.
Finally, Assumptions~\ref{hyp:A}--\ref{hyp:A0bound} are fulfilled and
the conclusion follows from Corollary~\ref{coro:CVdom}.
\end{proof}

\subsection{Case of distinct domains}

\pb{When domains $D_s$ are possibly distinct, the convergence result will follow from Theorem~\ref{the:main}.
We should therefore verify the conditions under which the latter holds.
Checking Assumptions~\ref{hyp:A}--\ref{hyp:A0bound} follows the same lines as in Section~\ref{sec:CVdom-fct} and is relatively easy.
Assumption~\ref{hyp:reg-domaine} will be kept as a standing assumption. 
The goal is therefore to provide a verifiable
condition under which Assumption~\ref{hyp:Jreg} holds. This condition is given as follows.

\begin{assum}
\label{hyp:dom-diff}
There exists $p\in \bN^*$ and $C\in L^2(E,\bR_+,\mu)$ such that for all $s\in E$ $\mu$-a.e. and all $x\in \dom(\partial f(s,\,.\,))$, 
$$\|\partial f_0(s,x)\|\leq C(s)(1+\|x\|^p)$$
and $\cZ_A(2p)\neq\emptyset$.
Moreover, $\dom(\partial f(s,\,.\,))$ is closed $\mu$-a.e.
\end{assum}
\smallskip

In order to verify that the above condition is indeed sufficient to ensure that Assumption~\ref{hyp:Jreg} holds, we need the following lemma.}

\begin{lemma}
\label{lem:regSousdif}
 Let $g:\cH\to (-\infty,+\infty]$ be a proper lower semicontinuous convex function. Consider $x\in \cH$ and $\lambda>0$. Let $\pi$ be the projection of $x$ onto $\overline{\dom}(g)$. Assume that $\partial g(\pi)\neq\emptyset$. Then,
$
\|\mathrm{prox}_{\lambda g}(x)-\pi\|\leq 2\lambda \|\partial g_0(\pi)\|\,.
$
\end{lemma} 
\begin{proof}
When $x=\pi$, the result is standard \cite[Corollary 23.10]{bauschke2011convex} (and the factor~2 in the inequality can even be omitted). We assume in the sequel that $x\neq \pi$.
  Define $j = \mathrm{prox}_{\lambda g}(x)$, $\varphi = \partial g_0(\pi)$ and 
  \begin{equation}
q = \arg\min_{y\in H} g(\pi)+\ps{\varphi,y-\pi}+\frac{\|y-x\|^2}{2\lambda}\label{eq:q}
\end{equation}
where $H$ is the half-space $\{y\in \cH:\, \ps{y-\pi,x-\pi}\leq 0\}$. 
\pb{By the Karush-Kuhn-Tucker conditions,
there exists $\alpha\geq 0$ such that $\lambda \varphi = -q +x-\alpha(x-\pi)$ along with the complementary slackness condition $\alpha \ps{q-\pi,x-\pi}=0$}.
 Now as $\varphi\in \partial g(\pi)$
and $(x-j)/\lambda\in\partial g(j)$, it follows by monotonicity of $\partial g$ that
\begin{align*}
  0&\leq \ps{\lambda \varphi- x+j,\pi-j} \\
&= \ps{j-q ,\pi-j}+\alpha \ps{x-\pi,j-\pi}.
\end{align*}
As $\ps{x-\pi,j-\pi}\leq 0$, we have $0\leq \ps{j-q ,\pi-j}$ which in turn implies that $\|j-\pi\|\leq \|q-\pi\|$.
\pb{As $q\in H$, it is clear that $\|x-\pi\|\leq \|q-x\|$ and thus $\|q-\pi\|\leq \|q-x\|+\|x-\pi\|\leq 2 \|q-x\|$.
Putting all pieces together, $\|j-\pi\|\leq 2 \|q-x\|$. Recall the identity, $q-x = -\lambda \varphi-\alpha(x-\pi)$.
If $\alpha=0$, the $\|q-x\|=\lambda\|\varphi\|$ and the conclusion  $\|j-\pi\|\leq 2\lambda\|\varphi\|$ follows.
If $\alpha>0$, the complementary slackness condition yields $\ps{q-\pi,x-\pi}=0$. 
Replacing $q$ by its expression as a function of $\alpha$, this allows to write $\alpha=\ps{\lambda\varphi, \pi-x}/\|x-\pi\|^2$.
Hence, $q-x = -\lambda P\varphi$ where $P$ is an orthogonal projection matrix.
Therefore, $\|q-x\|\leq \lambda \|\varphi\|$ and again, the conclusion $\|j-\pi\|\leq 2\lambda\|\varphi\|$ follows.
}
\end{proof}
\smallskip

\begin{theo}
 Let Assumptions~\ref{hyp:step-iid},~\ref{hyp:reg-domaine},~\ref{hyp:integrand}, and~\ref{hyp:dom-diff} hold true. Suppose that
$\lambda_n/\lambda_{n+1}\to 1$ as $n\to \infty$. 
Consider the random sequence $(x_n)$ given by~(\ref{eq:algo-fct}) with weighted averaged sequence $(\overline x_n)$. 
Then, almost surely,  $(\overline x_n)$ converges weakly to a minimizer of $F$. 
\label{the:main-fct}
\end{theo}
\begin{proof}
When letting $A=\partial f$, the conditions in Assumptions~\ref{hyp:A}--\ref{hyp:zer}
are fulfilled by using the same arguments as in the proof of Theorem~\ref{the:CVdom-fct}.
Moreover, Assumption~\ref{hyp:dom-diff} implies that the uniform integrability condition
of Assumption~\ref{hyp:A0bound} holds.
To apply Theorem~\ref{the:main}, it is sufficient to verify the condition
of Assumption~\ref{hyp:Jreg} replacing $J_\lambda(s,\,.\,)$ with $\prox_{\lambda f(s,\,.\,)}$.
By Lemma~\ref{lem:regSousdif} and using $\Pi(s,x) \in D_s$, the following holds $\mu$-a.e.
\begin{align*}
  \|\mathrm{prox}_{\lambda f(s,\,.\,)}(x)-\Pi(s,x)\|&\leq 2\lambda \|\partial f_0(s,\Pi(s,x))\| \\
&\leq 2\lambda C(s) (1+\|\Pi(s,x)\|^p)\,.
\end{align*}
Let $x^*$ be an arbitrary point in $\cD$. One has
$\|\Pi(s,x)\|\leq \|x^*\| + \|\Pi(s,x)-\Pi(s,x^*)\|$ where we used the fact that
$x^*=\Pi(s,x^*)$ for all $s$ $\mu$-a.e.
By non-expansiveness of $\Pi(s,\,.\,)$, $\|\Pi(s,x)\|\leq \|x^*\| + \|x-x^*\|$.
Finally, there exists a constant $\alpha$ depending only on $p$ and $x^*$ such that
$\|\mathrm{prox}_{\lambda f(s,\,.\,)}(x)-\Pi(s,x)\|\leq \lambda \alpha C(s) (1+ \|x\|^p)$.
The conclusion follows from Theorem~\ref{the:main}.
\end{proof}


\subsection{A constrained programming problem}
\label{sec:constrained}

\pb{In this section, we provide an application example to the case of constrained convex minimization
over an finite intersection of closed convex sets.}

Let $(X_1,\dots,X_m)$ be a collection of non-empty closed convex subsets of $\cH=\bR^d$ where $d\in \bN^*$.
We consider the problem
\begin{equation}
\min\  F(x) \ \text{w.r.t. }{x\in X}\text{ where }X=\bigcap_{i=1}^m X_i\label{eq:pb-contraint}
\end{equation}
where $F(x)=\int f(s,x)d\mu(s)$ for all $x\in \cH$. 
Consider a random sequence $(I_n)$ on $\{0,1,\dots,m\}$ independent of~$(\xi_n)$, with distribution $p_i = \bP(I_n=i)$ 
for every $i\in \{0,1,\dots,m\}$. Consider the iterations
\begin{equation}
  \label{eq:algo-contraint}
  x_{n+1} =\left\{
    \begin{array}[h]{ll}
      \mathrm{prox}_{\lambda_n f(\xi_{n+1},\,.\,)}(x_n) & \text{if } I_{n+1}=0\\
      \mathrm{proj}_{X_{I_{n+1}}}(x_n) & \text{otherwise.}
    \end{array}
\right.
\end{equation}
\pb{Let us briefly discuss the algorithm. At each time $n$, the iteration either consists in applying
the proximity operator of $f(\xi_{n+1},\,.\,)$ or a projection. The choice is random, the former being 
applied when the r.v. $I_{n+1}$ is zero, the latter being applied otherwise. The value $p_0$ represents
the probability that the proximity operator of $f(\xi_{n+1},\,.\,)$ is applied.
On the opposite, when $I_{n+1}>0$, a certain set is further picked at random, and projection onto that set 
is applied.

\begin{remark}
Instead of applying \emph{either} the proximity operator
of $f(\xi_{n+1},\,.\,)$ or a projection, one could think of applying \emph{both} successively, in the flavor of 
Passty's algorithm~\cite{passty1979ergodic}. Although it is out of the scope of this paper,
the corresponding algorithm may be analyzed using similar principles.
\end{remark}
}

\begin{assum}
\hfill
  \begin{enumerate}[(i)]
  \item The sets $X_1,\dots,X_m$ are boundedly linearly regular in the sense of~(\ref{eq:linearly-reg}) and $X=\cap X_i$ is non-empty.
  \item $f:E\times \cH\to \bR$ is a normal convex integrand and $f(\,.\,,x)$ is integrable for each $x\in \cH$.
  \item 
A solution to~(\ref{eq:pb-contraint}) exists and any solution $x^\star$
satisfies $|\partial f(\,.\,,x^\star)|\in  L^2(E,\bR,\mu)$.
\item There exists $p\in\bN^*$ and a solution $x^\star_p$ 
such that $|\partial f(\,.\,,x^\star_p)|\in  L^{2p}(E,\bR,\mu)$.
\item There exists $C\in L^2(E,\bR_+,\mu)$ such that for any $x\in \cH$,
$\|\partial f_0(s,x)\|\leq C(s)(1+\|x\|^p)$ $\mu$-a.e.
  \end{enumerate}
\label{hyp:constraint}
\end{assum}

\begin{theo}
  Let Assumptions~\ref{hyp:step-iid} and \ref{hyp:constraint} hold. Consider the iterates $(x_n)$ given by~(\ref{eq:algo-contraint})
with weighted averaged sequence $(\overline x_n)$ where the random sequence $(I_n)$ is is defined above.
Assume that $p_i>0$ for all $i\in \{0,1,\dots,m\}$ and let $\lambda_n/\lambda_{n+1}\to 1$ as $n\to \infty$.
Then, almost surely, $(\overline x_n)$ converges in average to a solution to~(\ref{eq:pb-contraint}).
\end{theo}

\begin{proof}
We introduce the random sequence $\tilde \xi_{n}=(\xi_n,I_n)$ on the set $\tilde E = E\times \{0,1,\dots,m\}$
equipped with the corresponding product $\sigma$-algebra.
We denote by $\nu=\mu\otimes (\sum_{i=0}^m p_i\delta_i)$ the probability distribution of $\tilde \xi_n$
where $\delta_i$ stands for the Dirac measure at $i$.
For all $\tilde s=(s,i)$ in $\tilde E$ and $x\in \cH$, define
$$
\tilde f(\tilde s,x) = f(s,x)\chi_{\{0\}}(i) + \sum_{j=1}^m \iota_{X_j}(x)\,\chi_{\{j\}}(i)
$$
where $\chi_C$ is the characteristic function of a set $C$ (equal to 1 on that set and zero outside) and
$\iota_C$ is the indicator function of a set $C$ (equal to 0 on that set and $+\infty$ outside).
We use the convention $0\times (+\infty)=0$. The iterations~(\ref{eq:algo-contraint})
also write
$$
x_{n+1} = \prox_{\lambda_n \tilde f(\tilde \xi_{n+1},\,.\,)}(x_n)\,.
$$
\pb{The rest of the proof consists once again in checking the conditions of application of Theorem~\ref{the:main},
when $E$, $\xi$, $A$ are respectively replaced by $\tilde E$, $\tilde \xi$, $\partial \tilde f$.
\smallskip

\emph{Checking Assumptions~\ref{hyp:A}, \ref{hyp:Abar} and~\ref{hyp:zer}}.
We first make the following observations.\\
}
(i)  $\tilde f$ is a normal convex integrand on $\tilde E\times \cH\to \bR$.\\
(ii) As $f(\,.\,,x)$ is integrable for any $x$, it follows that $F=\int f(\,.\,,x)$ is 
proper, convex and continuous.
Since $p_i>0$ for all $i$, the integral functional $\tilde F(x) = \int \tilde f(\,.\,,x)d\nu$ is equal to
$$
\tilde F(x) = p_0\, F(x) + \iota_X(x) 
$$
where $X = \bigcap_{i=1}^m X_i$.
\pb{As $X$ is a non-empty closed convex set and  $\dom(F)=\cH$,} it follows that
$\tilde F$~is proper and lower semicontinuous.\\
(iii) 
\pb{Let $N_C(x)$ denotes the normal cone of a closed convex set $C$ at point $x$. By the same argument,}
$$
\partial \tilde F(x) = p_0\partial F(x) + {N}_X(x)\,.
$$
Moreover, for any $\tilde s=(s,i)$,
\begin{equation}
\partial \tilde f(\tilde s,x) = \partial f(s,x)\chi_{\{0\}}(i) + \sum_{j=1}^m N_{X_j}(x)\,\chi_{\{j\}}(i)\label{eq:partial-tilde}
\end{equation}
and it follows that
$$
\int \partial \tilde f(\,.\,,x)d\nu = p_0 \int \partial f(\,.\,,x)d\mu + \sum_{i=1}^m N_{X_i}(x)\,.
$$
\pb{By Assumption~\ref{hyp:constraint}(i), the sets $X_1,\dots, X_m$ are linearly regular.
By \cite[Theorem 3.6]{bauschke1999strong}, this implies that $\sum_{i=1}^m N_{X_i}(x)=N_X(x)$.
Moreover, as $F$ is everywhere finite, $\int \partial f(\,.\,,x)d\mu = \partial F(x)$ by \cite{rockafellar1982interchange}.}
We conclude that for every $x\in \cH$,
\begin{equation}
\int \partial \tilde f(\,.\,,x)d\nu = \partial \tilde F(x)\,.\label{eq:Aumann-tilde}
\end{equation}
(iv) The minimizers of $\tilde F$ are the solutions to~(\ref{eq:pb-contraint}) and vice-versa.
In particular, $\tilde F$ admits minimizers. Let us prove that each minimizer $x^\star$ belongs to 
$\cZ_{\partial \tilde f}(2)$. 
By Fermat's rule, $0\in \partial \tilde F(x^\star)$.
Using successively~(\ref{eq:Aumann-tilde}) and~(\ref{eq:partial-tilde}), there exists 
$\phi\in S_{\partial f}(x^\star)$ and $(u_1,\dots,u_m)\in N_{X_1}(x^\star)\times\dots\times N_{X_m}(x^\star)$ such that
$0 = p_0 \int \phi d\mu+\sum_{i=1}^mp_i u_i$. Define for any $(s,i)\in \tilde E$,
$\tilde\phi(s,i) =  \phi(s)\chi_{\{0\}}(i) + \sum_{j=1}^m u_j\,\chi_{\{j\}}(i)$.
Clearly, $\tilde\phi(s,i)\in \partial \tilde f((s,i),x^\star)$ and $\int\tilde\phi d\nu=0$.
By Assumption~\ref{hyp:constraint}(iii), $\int \|\tilde\phi\|^{2}d\nu<+\infty$.
Therefore, $x^\star\in \cZ_{\partial \tilde f}(2)$.

We have checked that the four conditions in Assumption~\ref{hyp:integrand} are fulfilled
when $f$ and $F$ are respectively replaced by $\tilde f$ and $\tilde F$.
Now set $A=\partial \tilde f$. Using the same arguments as in the proof of Theorem~\ref{the:CVdom-fct},
the operator $A$ satisfies the conditions in Assumptions~\ref{hyp:A}, \ref{hyp:Abar} and~\ref{hyp:zer}.
\smallskip

Assumption~\ref{hyp:step-iid} being granted, it remains to check that $A=\partial \tilde f$
fulfills Assumptions~\ref{hyp:A0bound}, \ref{hyp:reg-domaine} and~\ref{hyp:Jreg}.
\smallskip

\pb{\emph{Checking Assumptions~\ref{hyp:A0bound} and \ref{hyp:reg-domaine}.}
By Equation~(\ref{eq:partial-tilde}), $\partial f_0(s,x)\chi_{\{0\}}(i)\in \partial \tilde f(\tilde s,x)$.
Therefore, $\|\partial \tilde f_0(s,x)\|\leq \|\partial  f_0(s,x)\|$. By Assumption~\ref{hyp:constraint}(v), 
the uniform integrability condition in Assumption~\ref{hyp:A0bound} is fulfilled. 
Using the linear regularity of the sets $X_1,\dots, X_m$,
Assumption~\ref{hyp:reg-domaine} is satisfied when substituting $D_s$ with $\dom(\partial \tilde f(s,\,.\,))$.
\smallskip

\emph{Checking Assumption~\ref{hyp:Jreg}.}
We finally check that $A=\partial \tilde f$ fulfills Assumption~\ref{hyp:Jreg}.}
Let $p\in \bN^*$ and $x^\star_p$ be defined as in Assumption~\ref{hyp:constraint}(iv).
Following the exact same line as above, one can construct $\tilde\phi$ such that 
$\tilde\phi(s,i)\in \partial \tilde f((s,i),x^\star_p)$,  $\int\tilde\phi d\nu=0$
and $\int \|\tilde\phi\|^{2p}d\nu<+\infty$. Therefore $\cZ_{\partial\tilde f}(2p)\neq\emptyset$.
Denote by $\tilde J_\lambda(\tilde s, x) = \prox_{\lambda \tilde f(\tilde s, \,.\,)}(x)$ and $\tilde \Pi(s,x)$
the projection of $x$ onto the domain of $\partial \tilde f(\tilde s, \,.\,)$.
For any $\tilde s=(s,i)$, one has $\tilde J_\lambda(\tilde s, x)-\tilde \Pi(\tilde s, x)=0$ if $i\geq 1$.
When $i=0$, $\tilde J_\lambda(\tilde s, x)=\prox_{\lambda f(s,\,.\,)}(x)$ and $\tilde \Pi(\tilde s, x)=x$. Thus,
$\frac 1\lambda \|\tilde J_\lambda(\tilde s, x)-\tilde \Pi(\tilde s, x)\|\leq \|\partial f_0(s,x)\|$
which is no larger that $C(s)(1+\|x\|)$. As $C$ is square-integrable, 
we conclude that the operator $A=\partial \tilde f$ fulfills Assumption~\ref{hyp:Jreg}.

By Theorem~\ref{the:main}, the iterates~(\ref{eq:algo-contraint}) almost surely converge weakly in average
to a zero of $\partial \tilde F$. As zeroes of  $\partial \tilde F$ coincide with solutions to~(\ref{eq:pb-contraint}),
the proof is complete.
\end{proof}

\pb{
\section{Conclusion}

In this paper, we introduced a stochastic proximal point algorithm for random maximal 
monotone operators
and proved the almost sure weak ergodic convergence of the algorithm toward a zero of the
Aumann expectation of the latter random operators.
The paper suggests that, by using the concept of random monotone operators, it is possible
to easily derive stochastic versions of different fixed point algorithms and to prove their almost 
sure convergence. This idea can be extended to provide stochastic counterparts of other algorithms:
the forward-backward algorithm which involves both implicit and explicit calls of the 
operators~\cite{bauschke2011convex}, Passty's algorithm~\cite{passty1979ergodic} or the 
Douglas-Rachford algorithm~\cite{lions1979splitting}.
Other important questions include the derivation of convergence rates.
Although a complexity analysis of the stochastic proximal point algorithm~(\ref{eq:algo-intro}) seems out of reach in 
the general setting, it would be important to address such an analysis in the special
case of convex programming~(\ref{eq:intro-ite-prox}). The paper \cite{nemirovski2009robust}
follows such an approach, in the case where the convex function are used explicitely.
An interesting perspective would be to extend the method to the case to the stochastic proximal point algorithm.
An alternative is to investigate asymptotic convergence rates, as in~\cite{ryu2014stochastic}.
Finally, the relaxation of the i.i.d. assumption over 
the random monotone operators would be an important problem in future works.}

\section*{Acknowledgement}

\pb{The author also would like to thank the anonymous reviewers for
their comments and for pointing useful references.
The author is grateful to Walid Hachem for important suggestions 
which allowed to significantly improve the manuscript.  
}

\bibliographystyle{siam} 

\end{document}